\documentclass[12pt]{amsart}
\usepackage{amsmath,amssymb}
\usepackage{syntonly}
\usepackage{eufrak}
\usepackage{amsthm}
\usepackage{enumitem}
\usepackage{setspace}
\usepackage{graphicx}
\usepackage{graphics}
\usepackage{color}
\usepackage[dvipsnames]{xcolor}
\usepackage{hyperref}
\usepackage{youngtab}
\usepackage{lscape}
\usepackage{faktor}
\usepackage{hyperref}
\usepackage[normalem]{ulem}
\usepackage{bm}
\usepackage{afterpage}
\usepackage{mathrsfs}
\usepackage[all]{xy}
\xyoption{matrix}
\xyoption{arrow}
\usepackage{tikz}
\usepackage{tikz-cd}
\usetikzlibrary{calc,shadings,patterns}
\usetikzlibrary{matrix}
\usetikzlibrary{arrows}
\usetikzlibrary{matrix}
\usepackage{verbatim}
\usepackage{multirow}
\usetikzlibrary{decorations.pathreplacing}
\usepackage{enumitem,pifont,xcolor}
\definecolor{myblue}{RGB}{0,29,119}

\newtheorem{theorem}{Theorem}[subsection]
\newtheorem{proposition}[theorem]{Proposition}
\newtheorem{corollary}[theorem]{Corollary}
\newtheorem{lemma}[theorem]{Lemma}
\newtheorem{lemma-definition}[theorem]{Lemma-Definition}
\newtheorem{proposition-definition}[theorem]{Proposition-Definition}
\theoremstyle{definition}
\newtheorem{definition}[theorem]{Definition}
\newtheorem{example}[theorem]{Example}

\newtheorem{remark}[theorem]{Remark}

\newtheorem{notation}[theorem]{Notation}
\newtheorem{claim}{Claim}[theorem]

\newtheorem*{theorem*}{Theorem}

\makeatletter
\@addtoreset{equation}{section}
\makeatother

\newcommand{\Aa}{{\mathbb A}}

\newcommand{\filt}{Filt}

\newcommand{\op}{\operatorname{op}}

\DeclareMathOperator{\End}{End}
\DeclareMathOperator{\coker}{coker}
\DeclareMathOperator{\im}{Im}
\DeclareMathOperator{\findim}{fin.\! dim}
\DeclareMathOperator{\gordim}{gor.\! dim}
\DeclareMathOperator{\gldim}{gl\,\! dim}
\DeclareMathOperator{\domdim}{dom.\! dim}
\DeclareMathOperator{\codomdim}{co-dom.\! dim}%
\DeclareMathOperator{\pdim}{p\!\dim}
\DeclareMathOperator{\indim}{in.\!dim}
\DeclareMathOperator{\Hom}{Hom}

\DeclareMathOperator{\defect}{def}
\DeclareMathOperator{\del}{del}
\DeclareMathOperator{\rank}{rank}
\DeclareMathOperator{\rel}{rel}
\DeclareMathOperator{\soc}{soc}
\DeclareMathOperator{\topp}{top}
\DeclareMathOperator{\rad}{rad}

\DeclareMathOperator{\Top}{top}

\newcommand{\gt}[1]{{\textcolor{PineGreen}{GT:}\,\textcolor{PineGreen}{#1}}}

\newcommand{\co}[1]{\textcolor{blue}{Emre}\,\textcolor{BrickRed}{#1}}
\newcommand{\ho}[1]{\Hom_{\Lambda}(\cP,{#1})}

\DeclareMathOperator{\add}{add} 
 \DeclareMathOperator{\modd}{\textendash mod}%
 \DeclareMathOperator{\moddd}{mod-\!}
\DeclareMathOperator{\Gen}{Gen}%
\DeclareMathOperator{\Cogen}{Cogen}%
\DeclareMathOperator{\dist}{dist}%

\newtheorem{innercustomthm}{Theorem}
\newenvironment{customthm}[1]
  {\renewcommand\theinnercustomthm{#1}\innercustomthm}{\endinnercustomthm}

\newcommand{\cB}{{\mathcal B}}

\newcommand{\cF}{{\mathcal F}}

\newcommand{\cI}{{\mathcal I}}

\newcommand{\cK}{{\mathcal K}}

\newcommand{\cM}{{\mathcal M}}
\newcommand{\cN}{{\mathcal N}}

\newcommand{\cP}{{\mathcal P}}

\newcommand{\cQ}{{\mathcal Q}}

\newcommand{\cS}{{\mathcal S}}
\newcommand{\cT}{{\mathcal T}}
\newcommand{\cU}{{\mathcal U}}
\newcommand{\cV}{{\mathcal V}}

\newcommand{\cX}{{\mathcal X}}
\newcommand{\cY}{{\mathcal Y}}

\providecommand{\AMS}{$\mathcal{A}$\kern-.1667em%
\lower.25em\hbox{$\mathcal{M}$}\kern-.125em$\mathcal{S}$}

\begin{document}


\author{Emre SEN}
\author{Gordana Todorov}
\author{Shijie Zhu}

\title{defect invariant nakayama algebras}
\begin{abstract} We show that for a given Nakayama algebra $\Theta$, there exist countably many cyclic Nakayama algebras $\Lambda_i$, where $i \in \mathbb{N}$, such that the syzygy filtered algebra of $\Lambda_i$ is isomorphic to $\Theta$ and we describe those algebras $\Lambda_i$. We show, among these algebras, there exists a unique algebra $\Lambda$ where the defects, representing the number of indecomposable injective but not projective modules, remain invariant for both $\Theta$ and $\Lambda$. As an application, we achieve the classification of cyclic Nakayama algebras that are minimal Auslander-Gorenstein and dominant Auslander-regular algebras of global dimension three. Specifically, by using the Auslander-Iyama correspondence, we obtain cluster-tilting objects for certain Nakayama algebras. Additionally, we introduce cosyzygy filteed algebras and show that it is dual of syzygy filtered algebra.
\end{abstract}
\maketitle

\tableofcontents

\let\thefootnote\relax

\setcounter{section}{-1}

\section{Introduction}

M. Auslander's seminal result, known as the Auslander correspondence, brought new perspectives to the representation theory of algebras: there exists a bijection between the set of
Morita-equivalence classes of representation-finite Artin algebras $A$ and that
of Artin algebras $\Gamma$ with $\gldim\Gamma\leq 2\leq\domdim\Gamma$ where $\gldim\Gamma$, $\domdim\Gamma$ denote global and dominant dimensions of $\Gamma$ respectively. The bijection is given by $A\rightarrow \Gamma:=\End_{A}(M)$ for an additive generator $M$ of $\moddd A$. Hence, representation theory of $A$ can be carried into homological theory of $\Gamma$.

O. Iyama introduced the \emph{higher Auslander correspondence} in \cite{iyama} such that an Artin algebra $\Gamma$ satisfying the conditions
\begin{align}\label{equation of higher auslander algebras}
\gldim\Gamma\leq d+1\leq\domdim \Gamma
\end{align}
for $d\geq 1$ can be realized bijectively as $\Gamma:=\End_{A}(\cM)$, where $\cM$ is a $d$-cluster-tilting object in the category of finitely generated $A$-modules for an algebra $A$. Since then, it has become an active research field in the representation theory of Artin algebras. One problem in this field is the classification of $d$-cluster-tilting modules for a given class of algebras, or the construction and characterization of higher Auslander algebras, i.e., algebras satisfying (\ref{equation of higher auslander algebras}), within a given class of algebras.

We will focus the attention to Nakayama algebras: that is all indecomposable modules are uniserial. Let $\Lambda$ be a cyclic Nakayama algebra and $\cS(\Lambda)$ denote the complete set of representatives of isomorphism classes of socles of indecomposable projective $\Lambda$-modules. The syzygy filtered algebra $\bm\varepsilon(\Lambda)$, introduced in \cite{sen19}, is the endomorphism algebra of projective covers of Auslander-Reiten translates of elements of $\cS(\Lambda)$, i.e.,
\begin{align}\label{definition of epsilon}
\bm\varepsilon(\Lambda):=\End_{\Lambda}\cP\quad \text{where}\quad \cP=\bigoplus\limits_{S\in\cS(\Lambda)} P(\tau S).
\end{align}

As highlighted by CM. Ringel in \cite{rin2}[Appendices B,C] and in \cite{ringel2022linear}, this provides a strong tool to study the homological dimensions of $\Lambda$. Since the syzygy filtered algebra is again a Nakayama algebra, it enables us to relate the homological dimensions of the original algebra and syzygy filtered algebras inductively.  Moreover, the syzygy filtered algebra played a crucial role in constructing cyclic Nakayama algebras that are higher Auslander algebras in \cite{sen21}. In this work, we establish the reverse of \ref{definition of epsilon} by using \emph{defect} of an algebra $\Lambda$, denoted by $\defect\Lambda$,  which represents the count of indecomposable non-isomorphic injective but non-projective modules \cite{rene}.

We state our first result.
\begin{customthm}{A}\label{emre theorem main}
Let $\Theta$ be a cyclic or linear (not necessarily connected) Nakayama algebra of rank $n$,  $\left\{P_1,\ldots,P_n\right\}$ be the complete ordered set of representatives of indecomposable projective $\Theta$-modules. Let $w=(w_1,\ldots,w_n)$ be any integer sequence such that $w_i\geq \defect P_i$, where $\defect P_i$ denotes the number of proper injective quotients of $P_i$. Then,
\begin{enumerate}
    \item\label{emre thm main part 1} if $\sum w_i>\defect\Theta$, then there are countably many cyclic Nakayama algebras $\Lambda_j$ $j\geq 1$ such that $\bm\varepsilon(\Lambda_j)\cong\Theta$ and $\defect\Lambda_j=\sum w_i$,
    \item\label{emre thm main part 2} if $\sum w_i=\defect\Theta$, then $\Lambda$ is unique.
\end{enumerate}
\end{customthm}

Using Theorem \ref{emre theorem main}, we prove the following theorem about dominant dimensions of Nakayama algebras.

\begin{customthm}{B}\label{emre thm domdim}
  Let $\Lambda$, $\Theta$ be Nakayama algebras such that $\bm\varepsilon(\Lambda)\cong\Theta$. Suppose that $\defect\Theta\neq 0$, then,
  \begin{enumerate}
      \item\label{emre thm domdim p1} if $\defect\Lambda>\defect\Theta$, then $\domdim\Lambda\leq 2$.
      \item\label{emre thm domdim p2} if $\defect\Lambda=\defect\Theta$, then $\domdim\Lambda=\domdim\Theta+2$.
  \end{enumerate}
\end{customthm}

\begin{figure}[h]
\begin{center}
$\xymatrixcolsep{5pt}\xymatrixrowsep{10pt}\xymatrix{\cdots\ar@/^1pc/[r]^{\varepsilon}&\Lambda \ar@/^1pc/[rr]^{\varepsilon}\ar@/^1pc/[l]^{\varepsilon^{-1}}\ar[dd] && \Lambda'\ar@/^1pc/[rr]^{\varepsilon}\ar@/^1pc/[ll]^{\varepsilon^{-1}}\ar[dd] && \Lambda''\ar@/^1pc/[r]^{\varepsilon}\ar@/^1pc/[ll]^{\varepsilon^{-1}}\ar@/^1pc/[ll]^{\varepsilon^{-1}}\ar[dd]&\cdots\ar@/^1pc/[l]^{\varepsilon^{-1}} \\\\
\text{Rank of Algebra:}&\rank \Lambda&&\rank \Lambda-\defect\Lambda&&\rank\Lambda-2\defect\Lambda\\
\text{Dominant Dimension:}&k&&k-2&&k-4}$
\end{center}
\caption{Defect Invariant Nakayama Algebras}
\label{figure 1}
\end{figure}
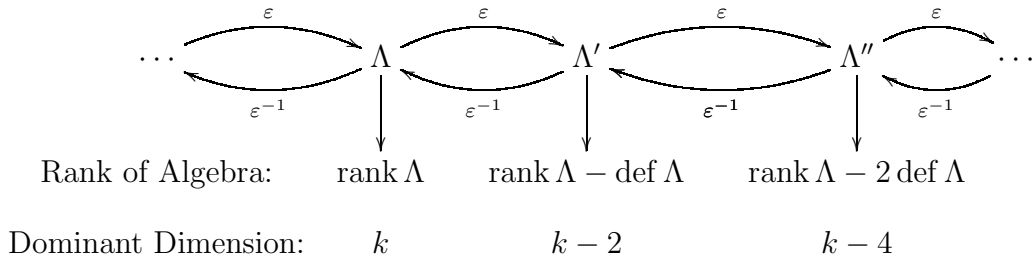

A direct consequence of Theorem \ref{emre theorem main} \ref{emre thm main part 2} and Theorem \ref{emre thm domdim} \ref{emre thm domdim p2} is: for a given Nakayama algebra $\Theta$ with nonzero defect, there exists unique cyclic Nakayama algebra $\Lambda$ such that $\domdim \Lambda=\domdim\Theta+2$ (Figure \ref{figure 1}). Hence, this enables us to construct Nakayama algebras characterized by their dominant dimension. 
For instance, a finite-dimensional algebra is called minimal Auslander-Gorenstein if its Gorenstein dimension is bounded by its dominant dimension \cite{iyamasolberg}. In the same vein, dominant Auslander-Gorenstein algebras were introduced in \cite{CIM}, where the algebra is Gorenstein and the injective dimension of each indecomposable projective module is bounded by its dominant dimension. If it has finite global dimension, it is called  dominant Auslander-regular algebra. We illustrate them in the following chart \ref{figue 2}.

\begin{figure}[h]
\begin{center}
\xymatrixcolsep{5pt}\xymatrixrowsep{2pt}\xymatrix{\text{Higher Auslander Algebras}\ar@{^{(}->}[rr]&& \text{Minimal Auslander-Gorenstein Algebras}\\
\gldim A\leq\domdim A\ar@{^{(}->}[dd] && \indim A\leq\domdim A \ar@{^{(}->}[dd] \\\\
\text{Dominant Auslander-Regular Algebras}\ar@{^{(}->}[rr] && \text{Dominant Auslander-Gorenstein Algebras}\\
\indim P_i\leq\domdim P_i,\,\indim P_i<\infty && \indim P_i\leq\domdim P_i \\
\forall P_i\,\,\text{indecomposable projective} &&\forall P_i\,\,\text{indecomposable projective}}
\end{center}
\caption{Nakayama algebras characterized by dominant dimension}
\label{figue 2}
\end{figure}
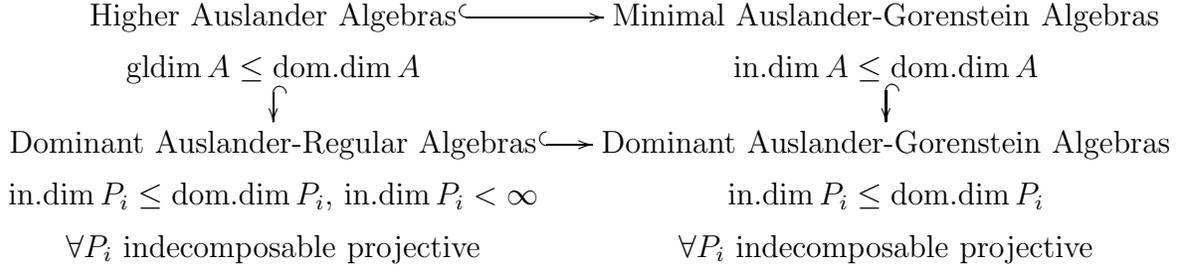

We prove the following:

\begin{customthm}{C}\label{emre thm minimal aus. gor}
    Let $\Lambda,\Theta$ be Nakayama algebras such that $\bm\varepsilon(\Lambda)\cong\Theta$ and $\defect\Lambda=\defect\Theta$. If $\Theta$ is a dominant Auslander-Gorenstein algebra, $\Lambda$ is as well.
\end{customthm}

To provide a complete classification of Nakayama algebras with infinite global dimension that are minimal Auslander-Gorenstein, it suffices to construct $2$-Auslander-Gorenstein algebras as described in Section \ref{subsection minimal Auslande-Gorenstein}. 
Additionally, we provide a complete classification of cyclic Nakayama algebras with global dimensions $3$ and $4$ which are higher Auslander. Through this classification, we determine which cyclic Nakayama algebras have $2$ and $3$ cluster tilting modules.

However, even for Nakayama algebras to classify all dominant Auslander-Gorenstein algebras that are not minimal Auslander-Gorenstein, it turns out that Theorem \ref{emre thm domdim} \ref{emre thm domdim p2} alone is not sufficient. Therefore, we focus on certain classes of dominant Auslander-Gorenstein and dominant Auslander-regular algebras. Based on Theorem \ref{emre thm domdim}, we classify all cyclic dominant Auslander-regular algebras of global dimension three in Section \ref{subsection dominant Auslander}. On the other hand, we give a construction of dominant Auslander-regular algebras of dominant dimension one and two in Section \ref{subsection dominant Auslander}. Additionally, by Theorem \ref{emre thm domdim} \ref{emre thm domdim p1}, we show that the dominant dimension of dominant Auslander-Gorenstein algebras of infinite global dimension is even.

\subsection{Organization of the paper}
In Section \ref{section prelim}, we recall the basic properties of Nakayama algebras. In Section \ref{section_epsilon}, we provide details of the syzygy filtered algebra construction also referred to as $\bm\varepsilon$-construction and study its functorial and module theoretical properties. Then, in Section \ref{section reverse}, we consider the reverse of the $\bm\varepsilon$-construction and prove Theorem \ref{emre theorem main}. In Section \ref{section higher homological}, we study the homological properties of algebras, prove Theorems \ref{emre thm domdim}, \ref{emre thm minimal aus. gor}, and provide a complete classification of certain classes of algebras. In Section \ref{sec:duality}, we consider cosyzygy filtered algebras and demonstrate that they are dual to syzygy filtered algebras.

\section{Preliminaries on Nakayama algebras}\label{section prelim}

We recall basic  properties of Nakayama algebras and the $\bm\varepsilon$-construction. This is a process of constructing so called ``syzygy filtered algebra'' which is closely related to the original algebra,
with respect to homological properties.

\subsection{Nakayama algebras}
Nakayama algebras are finite dimensional algebras for which  all indecomposable projective and  injective modules are uniserial.  The category of right finitely generated modules over a Nakayama algebra $\Lambda$ will be denoted by $\moddd\Lambda$. Throughout the paper $\tau$ will denote Auslander-Reiten translate.

\begin{remark}\label{emre remark labels may25} Let $\Lambda$ be a Nakayama algebra. Then $\Lambda\cong \mathbb{K}Q/I$, where $Q$ is a quiver, cyclic or linear, with $n$ vertices, $\mathbb{K}Q$ is the path algebra and $I$ is an admissible ideal of relations. We label the vertices of quiver $Q$ by indices $1,2\cdots, n$, such that
 the simple modules satisfy $S_{i+1}\cong \tau S_i$ for $1\leq i\leq n-1$, and $\tau S_{n}=\begin{cases} 0 &\text{if $\Lambda$ is linear} \\ S_1 &\text{  if $\Lambda$ is cyclic.}\end{cases}$
\end{remark}

{\color{black}
In fact, all the indecomposable modules over a Nakayama algebra are uniserial. Hence the projective cover and injective envelope of any indecomposable module is indecomposable. We will heavily use the following lemma about indecomposable modules over Nakayama algebras. Of particular importance for homological dimensions of Nakayama algebras are  minimal projective and minimal injective modules as defined below. All of these results are either already proved in \cite{sen18} and \cite{sen19}, or are not hard to check.
\begin{definition}\label{minimal}  An indecomposable projective module $P$ is called \emph{minimal projective}  if $\rad P$ is not projective.
An indecomposable injective module $I$ is called \emph{minimal injective}  if $I/\soc I$ is not injective.
\end{definition}

\begin{lemma}\label{lem:uniseriality}[Uniseriality]
Let $\Lambda$ be a Nakayama algebra and let $M$ and $N$ be indecomposable $\Lambda$-modules. 
\begin{enumerate}
\item If there exists $L\neq 0$ that is a submodule of both $M$ and $N$, then either $M$ is a submodule of $N$ or $N$ is a submodule of $M$.
\item If there exists $L\neq 0$ that is a quotient module of both $M$ and $N$, then either $N$ is a quotient module of $M$ or $M$ is a quotient module of $N$.
\end{enumerate}
\end{lemma}

%

   \begin{proposition}\label{prop:quadrilateral}
Let $\Lambda$ be a Nakayama algebra. 
\begin{enumerate}[label=\arabic*)]
\item Let  $0\to L\to M\to N\to 0$ be a non-split exact sequence of indecomposable modules. Then $\Top L\cong \tau \soc N$.
\item Let $S$ be a simple module. The indecomposable projective module $P(S)$ is non-injective if and only if $P(S)\cong \rad P(\tau^-S)$.
\item Let $S$ be a simple module. The indecomposable injective module $I(S)$ is non-projective if and only if $I(S)\cong I(\tau S)/\tau S$.
 \item Let $S$ be a simple $\Lambda$-module. Then $P(\tau S)$ is a projective-injective  $\Lambda$-module if and only if $P(S)$ is a minimal projective module.
  \item Let $S$ be a simple $\Lambda$-module. Then $I(S)$ is a projective-injective  $\Lambda$-module if and only if $I(\tau S)$ is a minimal injective module.
\end{enumerate}
\end{proposition}


\begin{proposition}\label{lem_defect_general_1}
Let $\Lambda$ be a cyclic Nakayama algebra. Then:
\begin{enumerate}[label=\arabic*)]
 \item $M\cong \rad P$ for some minimal projective module $P$ if and only if $M\cong I/\soc I$ for some minimal injective module $I$.
  \item Let $S$ be a simple $\Lambda$-module. If $P(S)$ is projective-injective, and $E$ is minimal injective quotient of $P(S)$ then
 $E/\soc E\cong \rad(P(\tau^{-1}S))$.
 \end{enumerate}
\end{proposition}


\subsubsection{Kupish Series $\&$ Gustafson Function}
We recall Kupisch series and Gustafson function. For integers $x$ outside the range 
 $\{1,\dots,n\}$ we will use symbol $[x]$ to denote 
 the integer $x$ (mod $n$) within the range $\{1,\dots,n\}$.
\begin{definition}\label{def:Gustafson}
Let $\Lambda\cong \mathbb{K}Q/I$ be Nakayama algebra  where $Q$ is either cyclic or connected linear quiver with vertices: 
 $\{1,\dots,n\}$. The Gustafson function is defined as: $g: \{1,\dots,n\}\to\{1,\dots,n\}$,  
 $g(i):=[i+\ell(P_i)-1]$ where $\ell(P_i)$  is the length of the projective cover of the simple $S_i$.
\end{definition}

\begin{lemma}
Let $\Lambda$ be cyclic Nakayama algebra. Then:\\
(a) Gustafson map defines map on simple $\Lambda$-modules as $\overline g(S)= \soc(P(S))$.  \\
(b) Let $P$ and $P'$ be indecomposable projective-injective $\Lambda$-modules. Then: \\
$(P\cong P')\iff (\topp P\cong \topp P'))\iff (\soc P\cong \soc P')$. 
\end{lemma}

}

\begin{definition}
    Let $\Lambda=\mathbb{K}Q/I$ be a connected Nakayama algebra where $Q$ is either cyclic or linear quiver with $n$ vertices, labeled as in Definition \ref{def:Gustafson}. Let $\{P_1,\ldots,P_n\}$ be the representatives of indecomposable projectives. The associated Kupisch series for $\Lambda$ is $(c_1,\ldots,c_n)$ where $c_i=\ell(P_i)$.
\end{definition}

\begin{lemma}\label{kupisch how to verify}
    Let $(c_1,\ldots,c_n)$ be sequence of positive integers. Then, this is an admissible sequence  for a Kupisch series of a Nakayama algebra, if there is a cyclic permutations of these integers such that
    \begin{enumerate}
        \item for each $1\leq i\leq n-1$, either $c_i=c_{i+1}$ or $c_i\leq c_{i+1}+1$.
        \item If $c_n\neq 1$, then either $c_1=c_n+1$ or $c_n\leq c_1$. If $c_n=1$, then algebra is linear and there is no additional condition.
    \end{enumerate}
\end{lemma}

\subsection{The defect, socle-set, base-set, and filtered projectives}
We now recall the properties of some of the notions which are fundamental for the rest of the paper.



\begin{definition}\label{defect} Let $P$ be an indecomposable projective module. Then \emph{defect of $P$}, denoted by $\defect(P)$ is the number of injective proper quotients of $P$. \emph{Defect of $\Lambda$}, denoted by $\defect(\Lambda)$  is the sum $\Sigma \defect(P)$ over a complete set of representatives of isomorphism classes of indecomposable projective $\Lambda$-modules.
\end{definition}

\begin{definition}\label{socle set} Let $\Lambda$ be a cyclic Nakayama algebra. Define:\\
$\cS(\Lambda)$ is a complete set of representatives of isomorphism classes of socles of indecomposable projective $\Lambda$-modules
and call it the \emph{socle-set}.\\
$\cS'(\Lambda):=\{\tau S\ | \ S\in \mathcal S(\Lambda)\}$
which is called the \emph{top-set} and the \emph{base-set} is defined as:
$\mathcal B(\Lambda):=\{B\ | \  \soc B\in \mathcal S(\Lambda), \topp B\in \mathcal S'(\Lambda), B \text{ is of minimal length}\}$.\\
$\mathcal P(\Lambda):=\{P(S')\ | \ \text{projective cover of } S'\in \cS'(\Lambda)\}$=$\{\text{filtered projectives}\}$ (Lemma \ref{projective}).\\
$\cS_f(\Lambda):=\{S\in \mathcal S(\Lambda)\ | \  S\cong \soc(P) \text{ for some filtered projective } P\in \cP(\Lambda) \}$ \\
$\cS_{nf}(\Lambda):=\{S\in \mathcal S(\Lambda)\ | \  S\ncong \soc(P) \text{ for any filtered projective } P\in \cP(\Lambda) \}$
\end{definition}

\begin{remark}\label{socle of projective-injective} Let $\Lambda$ be a cyclic Nakayama algebra. Then:\\
(1) A simple module $S$ is in the socle-set  $\cS(\Lambda)$ if and only if $S\cong\soc Q$ for some indecomposable projective-injective module $Q$.\\
(2) A simple module $S$ is in the top-set $\cS'(\Lambda)$ if and only if $S\cong \soc I$ for some minimal injective module $I$.\\
(3) Let $B \in \mathcal B(\Lambda)$. Then $B\neq 0$ and $B$ is a proper submodule of a projective-injective.\\
(4) Let $E$ be a minimal injective. Then $\soc E\in \cS'(\Lambda)$.
\end{remark}

\begin{definition}\label{rel}
Let $\Lambda\cong \mathbb{K}Q/I$ be cyclic Nakayama algebra, where $I$ is an admissible ideal. Let $\#\rel\Lambda=\#\{\text{irredundant set of relations which generate } I\}$.
\end{definition}
\begin{proposition}\label{numerical}\label{numerical_rel}
Let $\Lambda$ be a cyclic Nakayama algebra. Then:
\begin{enumerate}[label=\arabic*)]
\item $|\cS(\Lambda)|=|\cS'(\Lambda)|=|\cB(\Lambda)|=|\cP(\Lambda)|=\#\rel\Lambda$.
\item $\#\{\text{isomorphism classes of projective-injectives}\}=\#\rel\Lambda$
\item $\#\{\text{isomorphism classes of minimal projectives}\}=\#\rel\Lambda$
\item $\#\{\text{isomorphism classes of minimal injectives}\}=\#\rel\Lambda$
\item $\#\{\text{isomorphism classes of projectives filtered by } \cB(\Lambda)\}=\#\rel\Lambda$.
\end{enumerate}
\end{proposition}

\begin{proposition}\label{rank_rel_def}
$\rank\Lambda=\#\rel\Lambda+\defect{(\Lambda})$.
\end{proposition}
\begin{proof}
The number of isoclasses of indecomposable injective $\Lambda$-modules is equal to $\rank(\Lambda)$. The number of isoclasses of projective-injectives is equal to $\#\rel$. The number of isoclasses of indecomposable injectives which are not projective is $\defect(\Lambda)$.
\end{proof}

\begin{corollary} \label{emre defect zero implies selfinjective or semisimple}
    A Nakayama algebra $\Lambda$ has zero defect if and only 
    $\Lambda$ is selfinjective.
\end{corollary}
\begin{proof}
If $\Lambda$ is either selfinjective or semisimple, every projective $\Lambda$-module is injective. Hence $\defect\Lambda=0$ by Definition \ref{defect}. Assume $\defect\Lambda=0$, then, 
\begin{align*}
    \defect\Lambda=0 &\implies \rank\Lambda=\#\rel\Lambda\\
    &\implies \rank\Lambda=\#\{\text{isomorphism classes of projective-injectives}\}
\end{align*}
by Propositions \ref{rank_rel_def} and \ref{numerical}. So every projective $\Lambda$-module is injective which makes $\Lambda$ a selfinjective algebra.
  \end{proof}

{
As we will see in Section \ref{section_epsilon}, the base-set corresponds to the set of simple modules of the syzygy filtered algebras. We give some equivalent conditions for modules being in the base-set $\cB(\Lambda)$:
 \begin{proposition}\label{prop:B_eqv_def}
Let $\Lambda$ be a cyclic Nakayama algebra and $B$ an indecomposable module with $\soc B\in \cS(\Lambda)$. Then the following are equivalent:
\begin{enumerate}
\item [(i)]$B\in\cB(\Lambda)$.
\item [(ii)]$I(B)/B$ is non-injective and for each proper submodule $M\subseteq B$, $I(B)/M$ is injective.
\item [(iii)]$I(B)/\rad B$ is minimal injective.
\item [(iv)] $\ell (B)=\defect I(B)+1$.
\item [(v)]$\ell(B)$ equals the smallest natural number $n>0$ such that $\tau^{-n}\soc B\in\cS(\Lambda)$.
\end{enumerate}
\end{proposition}

\begin{proof} 
Denote by $S=\soc B$ and $I=I(B)$. \\
$1)\implies 2)$ Let $B\in\cB(\Lambda)$. Then $B\neq 0$ by Remark\ref{socle of projective-injective}(3). Since  $\topp B\in\mathcal S'(\Lambda)$ by Definition\ref{socle set} it follows that  $\soc(I(B))/B=\tau^{-}(\topp (B)$ by Uniseriality lemma \ref{lem:uniseriality}. Therefore $\soc (I(B)/B)\in \mathcal S(\Lambda)$. Therefore $I(B)/B$ is a submodule of a projective-injective module. Then $I(B)/B$ cannot be injective since that would imply that $I(B)/B$ is also projective injective. As a projective it would be a summand of $I(B)$ and therefore isomorphis to $I(B)$. This would imply that $B=0$ which is contradiction to Remark \ref{socle of projective-injective}. Therefore $I(B)/B$ is not injective. \\
Now, We want to show that $I(B)/M$ is injective for all  proper submodules $M$ of $B$. Let  $M\subsetneq B$ be of minimal length such that $I(B)/M$ is not injective. Then $I(B)/\rad M$ is a minimal injective. Then $\soc I(B)/\rad M\in \mathcal S'(\Lambda)$. Therefore $\topp M = \tau^{-}(\topp \rad M)=\soc (I(B)/\rad M) \in \mathcal S'(\Lambda)$. This is contradiction to the fact that $M$ is proper submodule of $B$. Therefore $I(B)/M$ is injective.
\\
$2)\implies 3)\implies 4)$ are straightforward.\\
$4)\implies 1)$ Assume $\defect (I)=d$. Then $I(\tau^{-i}S)$, for $0\leq i\leq d$,  are all the injective quotients of $I$. That is, $I(\tau^{-i}S)$ are non-projective for $1\leq i\leq d$ and $I(\tau^{-d-1}S)$ is projective-injective. Hence $\tau^{-i}S\not\in\cS'(\Lambda)$ for $0\leq i<d$ and $\tau^{-d}S\in\cS'(\Lambda)$. Clearly, $B$ having $\tau^i(S)$, $0\leq i\leq d$, as composition factors is a module in $\cB(\Lambda)$.\\
5)$\implies$ 1) By Definition \ref{socle set}.\\
1)$\implies$ 5) Assume $B\in\cB(\Lambda)$. So, $\soc B\in\cS(\Lambda)$, $\topp B\in\cS'(\Lambda)$ and $\topp B\cong \tau^{\ell(B)-1}\soc B$. $B\in\cB(\Lambda)$ implies all simple composition factors of $B/\soc B$ are not an element of $\cS(\Lambda)$. Hence $\tau^{-1}\topp B\cong\tau^{\ell(B)}\soc B\in\cS(\Lambda)$ makes $\ell(B)$ minimal. 
\end{proof}


\begin{proposition} \label{base-set}\label{prop_P(T)/B}
Let $\Lambda$ be a cyclic Nakayama algebra. Then:
\begin{enumerate}[label=\arabic*)]
\item\label{item1} Any element of the socle set $\mathcal S(\Lambda)$ is the socle of exactly one element of $\cB(\Lambda)$.
\item Any element of the top set $\mathcal S'(\Lambda)$ is the top of exactly one element of $\cB(\Lambda)$.
\item\label{item3} The simple composition factors of distinct $B\in \mathcal B(\Lambda)$ are disjoint.
\item Distinct elements of the base-set are Hom-orthogonal i.e. $\Hom_{\Lambda}\left(B,B'\right)\cong 0$ when $B\ncong B'$ and $\Hom_{\Lambda}(B,B)\cong\mathbb{K}$.
\item Each $B\in \cB(\Lambda)$ is a submodule of exactly one indecomposable projective-injective.
\item $B\in \cB(\Lambda)$ is simple $\Lambda$-module if and only if $B\in\cS'(\Lambda)\cap\cS(\Lambda)$.
\item $\sum_{B\in\cB(\Lambda)}\ell(B)=\rank\Lambda$.
\end{enumerate}
\end{proposition}

\begin{definition}  \label{Filt}
Let $\Lambda$ be cyclic Nakayama algebra and let $\cB(\Lambda)$ be the base-set as in Definition \ref{socle set}. Let $Filt(\cB(\Lambda)$ be the full subcategory of $\moddd\Lambda$ consisting of all modules which have filtration by modules from $\cB(\Lambda)$. 
\end{definition}

\begin{remark}\label{remark-Filt} Let $\Lambda$ be cyclic Nakayama algebra. Then:\\
(a) $Filt(\cB(\Lambda))$ is a wide subcategory of $\moddd\Lambda$ with modules in $\cB(\Lambda)$ as simple objects.\\
(b) Let $M\neq 0$ be an indecomposable $\Lambda$-module. Then $M$ is in $\mathcal B(\Lambda)$ if and only if $\topp M\in \mathcal S'(\Lambda)$ and $\soc M \in \mathcal S(\Lambda)$.
\end{remark}


\begin{lemma} \label{projective}Let $\Lambda$ be cyclic Nakayama algebra and let $P$ be a projective $\Lambda$-module. The following characterizations of $P$ are equivalent:
\begin{enumerate}[label=\arabic*)]
\item $P\in \cP(\Lambda)$ (Definition \ref{socle set}), i.e. $P$ is filtered projective
\item $P=P(S)$, i.e. $P$ is projective cover for some $S\in\cS'(\Lambda)$
\item $P=P(B)$, i.e. $P$ is projective cover  for some $B\in\cB(\Lambda)$
\item $P$ is filtered by $\cB(\Lambda)$.
\item $P$ is projective object in the wide subcategory $Filt(\cB(\Lambda))$
\end{enumerate}
\end{lemma}

We can interpret Propositions \ref{prop:quadrilateral} and \ref{prop:B_eqv_def} in terms of Kupisch series.
\begin{remark}\label{kupisch properties} 
Let $(c_1, c_2,\ldots c_n)$ be Kupisch series for Nakayama algebra $\Lambda$. Then,
\begin{enumerate}
    \item for each $1\leq i\leq n-1$
    \begin{enumerate}
        \item $c_i\geq 2$
        \item $c_i>c_{i+1}$ $\iff$ $P_i\supset P_{i+1}$ $\iff$$c_i=c_{i+1}+1$
        \item $c_i\leq c_{i+1}$ $\iff$ $P_i$ is minimal projective and $P_{i+1}$ is projective-injective with $\defect(P_{i+1})=c_{i+1}-c_{i}$
    \end{enumerate}
    \item $c_n=1$ $\iff$ $\Lambda$ is linear
\end{enumerate}
\end{remark}

\begin{lemma}\label{sunday lemma lengths of delta and cardinality} Let $(c_1,\ldots,c_n)$ be Kupisch series of a Nakayama algebra $\Lambda$. Then,
\begin{enumerate}[label=\arabic*)]
\item $\vert\cS(\Lambda)\vert=\vert\{c_i\mid c_i\leq c_{i+1}\}\vert$, and $\defect\Lambda=\vert\{c_{i+1}\mid c_i>c_{c_{i+1}}\}\vert$.
\item The lengths of elements of the base-set $\cB(\Lambda)$ is $c_{i+1}-c_i+1$ provided that $c_i\leq c_{i+1}$.
\end{enumerate}
\end{lemma}
\begin{proof}
    1) By Proposition \ref{numerical}, the cardinality of the socle set is equal to the number of minimal projective $\Lambda$-modules. $P_i$ is minimal if and only if $c_i=\ell(P_i)\leq\ell(P_{i+1})=c_{i+1}$ by Remark \ref{kupisch properties}, hence we get desired equality. Since $\{c_i\mid c_i\leq c_{i+1}\}\cup \{c_{i+1}\mid c_i>c_{c_{i+1}}\}=\{c_1,\ldots,c_n\}$, we conclude $\defect\Lambda=\vert\{c_{i+1}\mid c_i>c_{c_{i+1}}\}\vert$ by Proposition \ref{rank_rel_def}.\\
    2) By Remark \ref{kupisch properties}, $c_i\leq c_{i+1}$ implies that $\defect P_{i+1}=c_{i+1}-c_{i}$. By Proposition \ref{prop:B_eqv_def}, $\ell(B)=c_{i+1}-c_i+1$. 
\end{proof}

\begin{proposition}\label{T_P(T)_S}
Let $\Lambda$  be a cyclic Nakayama algebra and let $T$ be a simple $\Lambda$-module. Let $P(T)$ and $P(\tau T)$ be projective covers, $S=\soc(P(T))$ and $U=\soc(P(\tau T))$. Suppose $P(T)$ is a minimal projective. 
Then $U\cong \tau^{d+1} S$ 
where $d=\defect(P(\tau T))$.
\end{proposition}

\begin{proposition}\label{emre prop ccc1}
    Let $\Lambda$ be a cyclic Nakayama algebra. Let $B,C\in\cB(\Lambda)$ such that $\tau \soc B\cong \topp C$. Then, there exists a simple $\Lambda$-module $T$ such that $P(T)$ is projective-injective $\Lambda$-module with submodule $C$ and $P(\tau^{-1}T)$ is minimal projective with submodule $B$.
\end{proposition}
\begin{proof}
    By Proposition \ref{base-set} 5), there exists unique projective-injective module $P$ with submodule $C$. Let $T=\topp P$. Then, by Proposition \ref{prop:quadrilateral}, $P(\tau^{-1}T)$ is minimal projective. 
    Let $E$ be minimal injective quotient of $P$. Then, $P/C\cong E/\soc E$ is isomorphic to $\rad P(\tau^{-1}T)$. Hence $\soc P(\tau^{-1}T)\cong \soc B$ implies that $B$ is submodule of $P(\tau^{-1}T)$ by Uniseriality lemma \ref{lem:uniseriality}.
\end{proof}

\begin{proposition}\label{emre prop ccc2}
    Let $\Lambda$ be a cyclic Nakayama algebra. Let $B,C\in\cB(\Lambda)$ such that $\tau\soc B\cong\topp C $. If $S$ is a simple module such that  $P(S)$ and $P(\tau^jS)$ are projective $\Lambda$-modules with $B$ and $C$ as submodules, respectively, then there exist $k$ such that:
    \begin{enumerate}
        \item $C$ is submodule of $P(\tau^tS)$ where $k\leq t\leq j$
        \item $B$ is submodule of $P(\tau^sS)$ where $1\leq s<k$
    \end{enumerate}
\end{proposition}
\begin{proof}
    By Proposition \ref{emre prop ccc1}, there exists a simple $\Lambda$-module $T$ such that $P(T)$ is projective-injective having $C$ as submodule. Let $T=\tau^kS$. Since $C$ is submodule of $P(\tau^jS)$, by Uniseriality lemma \ref{lem:uniseriality}, $P(\tau^jS)\subseteq P(\tau^kS)=P(T)$. Therefore, every $P(\tau^tS)$ where $k\leq t\leq j$ have $C$ as submodule.
    
    On the other hand, $P(\tau^{-1}T)=P(\tau^{k-1}S)$ is minimal projective module by Proposition \ref{prop:quadrilateral}. By Proposition \ref{emre prop ccc1}, $B$ is submodule of     $P(\tau^{k-1}S)$. Since $B$ is submodule of $P(S)$, by Uniseriality lemma \ref{lem:uniseriality}, $P(\tau^{k-1}S)\subseteq P(S)$. Therefore, every $P(\tau^sS)$ where $1\leq s<k$ have $B$ as submodule. 
\end{proof}
\section{The $\varepsilon$-construction of Syzygy Filtered Algebras}\label{section_epsilon}

We now recall the definition of $\bm\varepsilon$-construction which was introduced in  \cite{sen18} and was fundamental in \cite{sen19},\cite{sen21}. The idea of this construction is to select the certan important projective modules and consider (a smaller algebra) of the endomorphisms of those modules. Those important projective modules $\cP(\Lambda)$ are precisely projective covers of the top-set $\cS'(\Lambda)$ as in Definition \ref{socle set}. The new constructed algebra is also called \emph{Syzygy Filtered Algebra} due to the fact that the category of modules of that algebra is equivalent to the category of all second syzygies \cite{sen18}.

\subsection{Definition and Functorial properties}
\begin{definition}\label{deffilteredalg} Let $\bm{\varepsilon}(\Lambda)$ be the endomorphism algebra of the direct sum of projective $\Lambda$-modules which are projective covers of elements of the top set $\cS'(\Lambda)$, i.e.
\begin{align*}
\mathcal P:= \bigoplus\limits_{S\in \cS'(\Lambda)}P(S) \quad \text{and} \quad
\bm{\varepsilon}(\Lambda):= \End_{\Lambda}(\mathcal P) = \End_{\Lambda}\left(\bigoplus\limits_{S\in \cS'(\Lambda)}P(S)\right).
\end{align*}
\end{definition}

This construction fits perfectly the classical set up of the endomorphism algebras and the functors between the corresponding module categories, which we now recall.


\begin{proposition}\label{artinalgebra}\label{functorial}
Let $\Lambda$ be an artin algebra,   $\emph{mod-}\Lambda$  the category of finitely presented right $\Lambda$-modules, $P$  a projective $\Lambda$-module,  $\Gamma= \End_{\Lambda}(P)$ and 
$\Hom_{\Lambda}(P,-):\emph{mod-}\Lambda \xrightarrow{}\emph{mod-} \Gamma$. Then:
\begin{enumerate}[label=\arabic*)]

\item 

 The functor $\Hom_{\Lambda}(P,-)$ induces an equivalence of categories
 $fp(\add P)\to \emph{mod-}\Gamma$, where $fp(\add P)$ is the full subcategory of $\emph{mod-}\Lambda $ consisting of the modules which are finitely presented by $\add P$.

 \item $(-\otimes_{\Gamma} P):\emph{mod-} \Gamma\xrightarrow{}\emph{mod-}\Lambda $ is left adjoint to $\Hom_{\Lambda}(P,-)$, it is right exact, fully faithful and
 induces inverse equivalence: $\emph{mod-} \Gamma \to fp(\add P)$.
 \item The composition $\Hom_{\Lambda}(P,-)\circ(-\otimes_{\Gamma}P)$ is the identity of $\emph{mod-}\Gamma$.
 \item $(-\otimes_{\Gamma}P)$ preserves projective covers.
\end{enumerate}
\end{proposition}




\begin{proposition} \label{diagram}Let $\Lambda$ be cyclic Nakayama algebra and let $\bm{\varepsilon}(\Lambda)$ be the associated syzygy filtered algebra.
Let $\cP=\oplus_{S\in\cS'} P(S)$. Then $Filt(\cB(\Lambda))= fp(\add\cP)$.
Then there is a commutative diagram:
\begin{gather}
\begin{aligned}
\xymatrixcolsep{10pt}
\xymatrix{   \moddd\Lambda\ar[rrrr]^{\!\Hom_{\Lambda}(\cP,-) \ \ } &&&&\moddd\bm{\varepsilon}(\Lambda)\ar[llll]^{(-\otimes_{\bm{\varepsilon}(\Lambda)}\cP)\ \ \ }\ar[lllld]^{(-\otimes_{\bm{\varepsilon}(\Lambda)}\cP)}.&\\
 Filt(\cB(\Lambda)) \ar[rrrru]^{\ _{\ _{\cong\ \ }}} \ar@{^{(}->}[u]
}
\end{aligned}
\label{diagram categorical equivalence}
\end{gather}
Then the following statements are true:
\begin{enumerate}[label=\roman*)]
\item $\Hom_{\Lambda}(\cP,-):\filt(\cB(\Lambda))\to\moddd\bm\varepsilon(\Lambda)$ is an equivalence with inverse $-\otimes_{\bm\varepsilon(\Lambda)}\cP$.
\item Let $M\in\emph{mod-}\bm{\varepsilon}(\Lambda)$. Then $M\cong\Hom_{\Lambda}(\cP,X)$ for some $X\in Filt(\cB(\Lambda))$.
\item Let $S \in \emph{mod-}\bm\varepsilon(\Lambda)$ be simple. Then $S\cong \Hom_{\Lambda}(\cP,B)$ for some $B\in\cB(\Lambda)$.
\item Let $P\in\emph{mod-}\bm\varepsilon(\Lambda)$ be projective. Then  $P\cong \Hom_{\Lambda}(\cP,\tilde P)$ where $\tilde P$ is $\cB(\Lambda)$-filtered projective $\Lambda$-module.
\item\label{EM filtered proj inj implies it is filtered in epsilon}  If $P$ is projective-injective $\Lambda$-module and $P\in Filt(\cB(\Lambda))$, then $\Hom_{\Lambda}(\cP,P)$ is projective-injective $\bm\varepsilon(\Lambda)$-module. 
\item Let $X\in Filt{\cB}(\Lambda)$. Then $X\cong (M\otimes_{\bm{\varepsilon}(\Lambda)}\cP)$ for some $M\in\emph{mod-}\bm{\varepsilon}(\Lambda)$.
\item Let $B\in{\cB(\Lambda)}$. Then $B\cong (S\otimes_{\bm{\varepsilon}(\Lambda)}\cP)$ for some  simple $S\in\emph{mod-}\bm{\varepsilon}(\Lambda)$.
\item Let $\tilde P$ be $\cB(\Lambda)$-filtered projective $\Lambda$-module. Then $\tilde P\cong (P\otimes_{\bm{\varepsilon}(\Lambda)}\cP)$ for some  projective  $P\in\emph{mod-}\bm{\varepsilon}(\Lambda)$.
\end{enumerate}
\end{proposition}

\begin{notation} Let $\bm\varepsilon(\Lambda)$ be the syzygy filtered algebra of a cyclic Nakayama algebra $\Lambda$.
Let $M$ be a  $\bm\varepsilon(\Lambda)$-module. We will often use notation $\Delta(M):=(M\otimes_{\bm{\varepsilon}(\Lambda)}\cP)\in Filt(\cB({\Lambda}))$, i.e., we will denote the functor $(-\otimes_{\bm{\varepsilon}(\Lambda)}\cP)$ by $\Delta$.
\end{notation}
\subsection{Relationship between defects of $\Lambda$ and its syzygy filtered algebra $\Theta$}

It was the reverse of $\bm\varepsilon$-construction that was used in \cite{sen21} for the description of Nakayama algebras which are higher Auslander algebras and in the Section \ref{section higher homological} it will be used also to describe families of algebras  satisfying higher homological properties.\\

Here we consider cyclic Nakayama algebras $\Lambda$ such that their syzygy filtered algebras $\bm\varepsilon (\Lambda)$ are isomorphic to a  given Nakayama algebra $\Theta$.

In general for any  Nakayama algebra $\Lambda$ particularly important sets of modules which will be used here, are: socle-set $\cS(\Lambda)$, top-set $\cS'(\Lambda)$, base-set $\cB(\Lambda)$, filtered projectives $\cP(\Lambda)$, socles of filtered projectives $\cS_f(\Lambda)$, the other socles of projectives $\cS_{nf}(\Lambda)$ (Definition \ref{socle set}). Also, the notion of defect (Def.\ref{defect}) and relations (Def.\ref{rel}) are used. In addition,  there is the basic relation $\rank(\Lambda)=\#\rel(\Lambda)+\defect\Lambda$ (Prop.\ref{rank_rel_def}).

\begin{proposition} \label{C_and_A}
Let $\Lambda$ and $\Theta$ be cyclic Nakayama algebras such that $\bm\varepsilon(\Lambda)\cong \Theta$.
\begin{enumerate}[label=\arabic*)]
\item $\#\rel(\Lambda)=\rank(\Theta)$
\item $|\cS_f(\Lambda)|=\#\rel(\Theta)$
\item $|\cS_{nf}(\Lambda)|=\defect(\Theta)$
\item $\defect(\Lambda)=$ number of non-filtered $\Lambda$-projectives
\item $\defect(\Lambda)=t+ \defect(\Theta)+r$, where $t=|\{\text{non-filt.proj.inj. }Q\ |\ \soc(Q)\in \cS_f(\Lambda)\}|$ and  $r=|\{\text{non-filt.proj non-inj}\}|$.
\end{enumerate}
\end{proposition}
\begin{proof}
1) This follows since $\#\rel(\Lambda)=|\cP(\Lambda)|$ by Prop.\ref{numerical_rel} and $\Theta=\End_{\Lambda}(\oplus_{P\in\cP(\Lambda)}P)$.

2) By Def.\ref{socle set}, $\cS_f(\Lambda)= \{S\in \cS\ |\ S\cong \soc(P)\text{ for some } P\in\cP(\Lambda)\}$, i.e. these are socles of filtered projectives.
Let $\cB_f(\Lambda): =\{B\in\cB(\Lambda)\ |\ \soc(B)\in \cS_f(\Lambda)\}$.
Any element of the socle set $\mathcal S(\Lambda)$ is the socle of exactly one element of $\cB(\Lambda)$ by Prop.\ref{socle set}.
Therefore $|\cS_f(\Lambda)|=|\cB_f(\Lambda)|$
and the functor Hom$_{\Lambda}(\cP, -)$ gives a bijection between $\cB_f(\Lambda)$ and the socle set $\cS(\Theta)$. Again, using Prop.\ref{numerical_rel} it follows that
$|\cS(\Theta)|=\#\rel(\Theta)$.

3) Since $\cS_{nf}(\Lambda)=\cS(\Lambda)\backslash \cS_f(\Lambda)$ it follows that $|\cS_{nf}(\Lambda)|=|\cS(\Lambda)|-|\cS_f(\Lambda)|$ =\\
$\#\rel(\Lambda)-\#\rel(\Theta)=\rank(\Theta)-\#\rel(\Theta)=\defect(\Theta)$.

4) By Def.\ref{defect} it follows that $\defect(\Lambda)=\#\{$injective, non-projective $\Lambda$-modules$\}$ and also
$\defect(\Lambda)=\#\{$projective, non-injective $\Lambda$-modules$\}$.
So
$\defect(\Lambda)=\rank(\Lambda)-\#\rel(\Lambda)=$
$$|\{\text{projective } \Lambda\text{-modules}\}|-|\{\text{filtered projective } \Lambda\text{-modules}\}|=
 |\{\text{non-filtered projective } \Lambda\text{-modules}\}|.$$

5) It follows from 4) that $\defect(\Lambda)=|\{\text{non-filtered projective }\Lambda\text{-modules}\}| =$
$$|\{\text{non-filt proj-inj }\Lambda\text{-modules}\}| + |\{\text{non-filt proj non-inj }\Lambda\text{-modules}\}| =$$
$$|\{\text{socles of non-filt proj-inj}\}| + |\{\text{non-filt proj non-inj }\Lambda\text{-modules}\}|.$$
Now, we will use the fact that
$|\{\text{socles of non-filt proj-inj}\}| =$
$$|\{S\in\cS_f(\Lambda), \text{ socle of non-filt proj-inj}\}|+|\{S\in \cS_{nf}(\Lambda), \text{ socle of non-filt proj-inj}\}|.$$
Notice that $\{S\in \cS_{nf}(\Lambda), \text{ socle of non-filt proj-inj}\}= \cS_{nf}(\Lambda)$ since $S\in\cS_{nf}(\Lambda)$ can only be socle for a non-filtered projective.
So this summand is equal to $\defect(\Theta)$ by 3).
Let $$t:=|\{\text{non-filt.proj.inj. }Q\ |\ \soc(Q)\in \cS_f(\Lambda)\}| \text{ and let } r:=|\{\text{non-filt.proj non-inj}\}|.$$
$$\defect(\Lambda)=t+\defect(\Theta)+r.$$
\end{proof}


 \begin{notation}\label{sequence} We will use term \emph{sequence} to mean cyclic ordering within AR-quiver - more precisely:
\emph{sequence of simples} means ordering of simples in the $\tau$-orbit according to the power of $\tau$,
\emph{sequence of projectives}  means ordering of projectives induced by the ordering of their tops,
\emph{sequence of $\Delta(S)$} means ordering induced by ordering of the simples $S$.
 \end{notation}

 \begin{remark}\label{non-filtered}Let $\Lambda$ be a cyclic Nakayama algebra. Let $B\in \cB(\Lambda)$. Then any projective $P(Z)$ between $P(\tau(\soc(B)))$ and $P(\topp(B))$,
 not including $P(\topp(B))$, is non-filtered projective since it's top $Z\notin \cS'$, so $P(Z)$ cannot be filtered, i.e. 
 for any composition factor $Z$ of $\rad B$, $P(Z)$ cannot be filtered.
\end{remark}

\begin{proposition} \label{length_defect} Let $\Lambda$ and $\Theta$ be cyclic Nakayama algebras such that $\bm\varepsilon(\Lambda)\cong \Theta$. Let $T$ be a simple $\Theta$-module, $P(T)$ and $P(\tau T)$ be projective covers in mod $\Theta$. 
Let $\Delta T$ be the corresponding $\Lambda$-module.
Then $\ell_\Lambda(\Delta T) \geq \defect_{\Theta}(P(\tau T))+1$ ($\ell$ is the length of the module).
\end{proposition}
\begin{proof}
\underline{Case 1.} Suppose $P(T)$ is not a minimal projective in mod-$\Theta$. Then $P(\tau T)=\rad P(T)$ and therefore $\defect(P(\tau T))=0$ since only projective-injectives might have nonzero defect. Since $\ell(\Delta(T))\geq 1$, it follows that $\ell(\Delta(T))\geq \defect P(\tau T)+1$.\\
 \underline{Case 2.} 
%
 Suppose $P(T)$ is a minimal projective in mod-$\Theta$. Then $P(\tau T)$ is projective injective by Prop.\ref{lem_defect_general_1}.
We will use the equivalence 
$\Delta =(- \otimes \cP):\moddd\Theta \to Filt(\cB(\Lambda))$ and Prop.\ref{T_P(T)_S}.
Let $S=\soc (P(T))$ and $U=\soc(P(\tau T))$. Then there are inclusions $\Delta (U)\to \Delta P(\tau T)= P(\Delta(\tau T))$ and
$\Delta (S)\to \Delta P(T)= P(\Delta(T))$. Here we used the fact that $\Delta$ preserves projective covers (Prop.\ref{artinalgebra}).
Let $d:=\defect(P(\tau T))$. Then $U\cong \tau^{d+1}S$ by Prop.\ref{T_P(T)_S}, and 
$\Delta(U)\cong \Delta(\tau^{d+1} S)$.
Consider the sequence of non-isomorphic modules from $\cB(\Lambda)$:
$$\{\Delta(U), \  \Delta(\tau^{-1} U), \  \Delta(\tau^{-2} U), \dots  ,\Delta(\tau^{-(d-1)} U), \  \Delta(\tau^{-(d)} U),\   \Delta(\tau^{-(d+1)} U)\cong \Delta(S)\}.$$
Each of $\Delta(\tau^{-i})$ for $1\leq i\leq (d+1)$ must be a submodule of some projective-injective $P(Z_i)$ for some simple module $Z_i$. Also $\Delta(U)$ is a submodule of $P(\Delta(\tau T))$ and $\Delta(S)$ is a submodule of $P(\Delta(T))$. This defines a sequence of non-isomorphic projectives:
$$\{P(\Delta(\tau T)), \ P(Z_1), \ P(Z_2), \ \dots,P(Z_{d-1}), \ P(Z_d), \ P(Z_{d+1})=P(\topp(\Delta(T)))\}$$
which further defines a sequence of non-isomorphic simple modules:
$$\{\topp(\Delta(\tau T), \  Z_1, \ Z_2, \ \dots , Z_{d-1}, \ Z_d, \ Z_{d+1}=\topp(\Delta(T)\}.$$
The simples 
$\{ Z_1, Z_2, \dots , Z_{d-1}, Z_d\}$ are among the composition factors of $\rad(\Delta(T))$. Therefore $d\leq \ell(\rad(\Delta(T)))$ and therefore $\defect(P(\tau T))+1=d+1\leq \ell(\Delta T)$.
\end{proof}

The filtered projective $\Lambda$-modules are well understood: they are  the images of projective $\Theta$-modules under the functor $\Delta=(-\otimes_{\Theta} \cP)$: essentially all simple $\Theta$-factors $S$ get replaced by $\Lambda$-modules $\Delta(S)$.
While the non-filtered projective $\Lambda$-modules which are directly related to the $\defect(P(\tau T))$
were already introduced in the previous proof as $P(Z_i)$,  
the following proposition will describe more precisely additional non-filtered projective $\Lambda$-modules which exist when $d=\defect(P(\tau T)) < \ell (\Delta(T))$.

\begin{proposition}   \label{length_defect_2} Let $\Lambda$ and $\Theta$ be Nakayama algebras such that $\bm\varepsilon(\Lambda)\cong \Theta$. Let $T$ be a simple $\Theta$-module, $P(T)$ and $P(\tau T)$ be projective covers,  $S=\soc (P(T))$ and $U=\soc(P(\tau T))$. Suppose $d=\defect(P(\tau T)) < \ell (\Delta(T))=m$. 
Then there is a sequence of indices 
$\{a_1, a_2, \dots , a_d\}\subset \{1,\ 2,\dots , m-1\}$ such that 
\begin{enumerate}[label=\arabic*)]

\item $\Delta(\tau^{-i}U) \hookrightarrow P(Z_{i})$ for non-filtered projective-injective $\Lambda$-modules $P(Z_{i})$ 
where $Z_{i}= (\tau^{-a_i}(\topp(\Delta(\tau T)))) $ are simple $\Lambda$-modules  for $1\leq i\leq d$.

\item For each $a_i\leq j \leq a_{i+1}$ there is a non-filtered projective non-injective $\Lambda$-module $P(Y_j)$ 
with the simple top $Y_j$ which fits in the sequence 
$\{\dots ,Z_{i}, Y_j, Z_{{i+1}}, \dots\}$.
Each such projective module $P(Y_j)$ has 
$\Delta(\tau^{-(i+1)}U)$ as a submodule.

\item For each $0\leq j \leq a_{1}$ there is a non-filtered projective $\Lambda$-module $P(Y_j)$ with the simple top $Y_j$ which fits in the sequence 
$\{\topp(\Delta(\tau T), Y_j, Z_{{1}}, \dots\}$.
Each such projective module $P(Y_j)$ has 
$\Delta(\tau^{-1}U)$ as a submodule.

\item For each $a_d\leq j \leq m-1$ there is a non-filtered projective $\Lambda$-module $P(Y_j)$ with the simple top $Y_j$  which fits in the sequence 
$\{ Z_{{d}}, Y_j, \dots,\topp(\Delta(\tau T)\}$.
Each such projective module $P(Y_j)$ has 
$\Delta(\tau^{-1}S)$ as a submodule.

\end{enumerate}
\end{proposition} 

\begin{proof} 1) Since $d=\defect(P(\tau T))$, we recall from the previous proof that there are $d$ non-isomorphic elements of the base-set $\cB(\Lambda)$ given as $\{\Delta(\tau^{-i}U\}_{i=1}^d$, the projective-injective modules $\{P(Z_i)\}_{i=1}^d$ which contain $\{\Delta(\tau^{-i}U\}_{i=1}^d$ as submodules for each $1\leq i\leq d$. Since the simple modules $\{Z_i\}_{i=1}^d$ are composition factors of $\rad(\Delta(T))$, they are not in $\cS'(C)$ and therefore the projectives $\{P(Z_i)\}_{i=1}^d$ are not filtered. 

Furthermore, the simple modules $\{Z_i\}_{i=1}^d$ are among the simple composition factors of $\Delta(T)$. 
We will denote $\widetilde{\tau T}\!\!: =\topp(\Delta(\tau T)))$  and  $\widetilde{T}\!\!: =\topp(\Delta(T))$. Notice that  
$\soc\Delta(T)=\tau^{-1}(\widetilde{\tau T})$ and  $\topp(\Delta(T)) =\tau^{-m}(\widetilde{\tau T})$
So composition factors of $\Delta(T)$ are
$$\{\tau^{-1}(\widetilde{\tau T}), \ \tau^{-2}(\widetilde{\tau T}), \ \dots \ , \tau^{-(m-1)}(\widetilde{\tau T})\}.$$
Therefore there is a sequence of indices 
$\{a_1, a_2, \dots , a_d\}\subset \{1,\ 2,\dots , m-1\}$ such that $Z_i = \tau^{-a_i}(\widetilde{\tau T})= \tau^{-a_i}(\topp(\Delta(\tau T)) \text{ for }  i\in \{1,\dots ,d\}$.

2) Suppose $a_i\leq j \leq a_{i+1}$. Let $Y_j = \tau^{-j}(\widetilde{\tau T})$ and let $P(Y_j)$ be the projective cover. Since each of $Y_j = \tau^{-j}(\widetilde{\tau T})$ is a composition factor of $\rad(\Delta T)$, the $Y_j$'s  are not in $\cS'(C)$ and therefore the projectives $P(Y_j)$ are not filtered.

 Notice that $\tau\soc\Delta(\tau^{-(i+1)}U)\cong \topp\Delta(\tau^{-i}U)$, and there exists projectives $P(Z_i)$ and $P(Z_{i+1})$ having $\Delta(\tau^{-i}U)$ and $\Delta(\tau^{-(i+1)}U)$ as submodules respectively. By Proposition \ref{emre prop ccc2}, $P(Y_j)$ has either $\Delta(\tau^{-i}U)$  or $\Delta(\tau^{-(i+1)}U)$ as submodule. Hence, we can choose $a_i$ so that $P(Z_i)$ is minimal projective, which makes $\Delta(\tau^{-(i+1)}U)$ submodule of $P(Y_j))$ by Proposition \ref{emre prop ccc1}.

3) Suppose $0\leq j \leq a_{1}$. Let $Y_j = \tau^{-j}(\widetilde{\tau T})$ and let $P(Y_j)$ be the projective cover. Since each of $Y_j = \tau^{-j}(\widetilde{\tau T})$ is a composition factor of $\rad(\Delta T)$, the $Y_j$'s  are not in $\cS'(C)$ and therefore the projectives $P(Y_j)$ are not filtered.

 Notice that $\tau\soc\Delta(\tau^{-1}U)\cong \topp\Delta(U)$, and there exists projectives $P(Z_1)$ and $P(\Delta(\tau T))$ having $\Delta(\tau^{-1}U)$ and $\Delta(U)$ as submodules respectively. By Proposition \ref{emre prop ccc2}, $P(Y_j)$ has either $\Delta(\tau^{-1}U)$  or $\Delta(U)$ as submodule. Hence, we can choose $a_1$ so that $P(Z_1)$ is minimal projective, which makes $\Delta(\tau^{-1}U)$ submodule of $P(Y_j))$ by Proposition \ref{emre prop ccc1}.

 4) Similar to 3).
\end{proof}

\begin{remark} Suppose $\bm\varepsilon(\Lambda)\cong \Theta$. The image under $\Delta$ of a minimal projective $\Theta$-module, need not be a minimal projective
$\Lambda$-module. Also, the image under $\Delta$ of a projective-injective $\Theta$-module, need not be a projective-injective $\Lambda$-module.
$\Lambda$-module.
\end{remark}

\subsection{Defect invariant Nakayama algebras}
\begin{definition}\label{defn_defect_inv}
Let $\Lambda$ be a cyclic Nakayama algebra and $\bm\varepsilon(\Lambda)$ its syzygy filtered algebra. We call $\Lambda$ defect invariant if $\defect \Lambda=\defect \bm\varepsilon(\Lambda)$.
\end{definition}

We want to emphasize that:

\begin{proposition}\label{prop_new_proj_is_inj_C}
Let $\Lambda$ be a cyclic Nakayama algebra, and $\Theta:=\bm\varepsilon(\Lambda)$. Then $\Lambda$ is defect invariant if and only if both of the following hold:\\
(i) For each simple module $S\notin \cS'(\Lambda)$, $P(S)$ is projective-injective;\\
(ii) $\Hom_{\Lambda}(\cP,P)$ is a projective-injective $\Theta$-module if and only if $P$ is a filtered projective-injective $\Lambda$-module.
\end{proposition}

\begin{proof} It follows from Proposition \ref{C_and_A} that
$$\defect(\Lambda)=t+ \defect(\Theta)+r,$$
where $t=|\{\text{non-filt.proj.inj. }Q\ |\ \soc(Q)\in \cS_f(\Lambda)\}|$ and  $r=|\{\text{non-filt.proj non-inj}\}|$.
So, $\defect(\Lambda)=\defect(\Theta)$ if and only if $r=|\{\text{non-filt.proj non-inj}\}|=0$ and $t=|\{\text{non-filt.proj.inj. }Q\ |\ \soc(Q)\in \cS_f(\Lambda)\}|=0$.

(i) Since $r=0$ it follows that all non-filtered projectives must be injective.

(ii) Since $t=0$ it follows that all non-filtered projective-injectives must have socles in $\cS_{nf}(\Lambda)$. Therefore the number of non-filtered projective injective $\Lambda$-modules is $|\cS_{nf}(\Lambda)|=\defect(\Theta)$ by Prop.\ref{C_and_A}.


$\#\rel(\Lambda)=$
$|\{\text{proj.inj.}\Lambda\text{-modules}\}|=$
$|\{\text{filt.proj.inj.}\Lambda\text{-modules}\}|+|\{\text{non-filt.proj.inj.}\Lambda\text{-modules}\}|$\\
Hom$_\Lambda(\cP,-)$ gives a bijection between $\{\text{filt.proj.}\Lambda\text{-modules}\}$ and $\{\text{proj.}\Theta\text{-modules}\}$.\\
Hom$_\Lambda(\cP,-)$ maps $\{\text{filt.proj.inj.}\Lambda\text{-modules}\}$ to $\{\text{proj.inj.}\Theta\text{-modules}\}$.
Therefore \\
$|\{\text{filt.proj.inj.}\Lambda\text{-modules}\}|\leq|\{\text{proj.inj.}\Theta\text{-modules}\}|$ and \\
$\#\rel(\Lambda)=|\{\text{filt.proj.inj.}\Lambda\text{-modules}\}|+|\{\text{non-filt.proj.inj.}\Lambda\text{-modules}\}|\leq$\\
$|\{\text{proj.inj.}\Theta\text{-modules}\}|+|\{\text{non-filt.proj.inj.}\Lambda\text{-modules}\}|=$
$\#\rel(\Theta) + \defect(\Theta)=\rank(\Theta)$.\\
Since $\#\rel(\Lambda)=\rank(\Theta)$ it follows that $|\{\text{filt.proj.inj.}\Lambda\text{-modules}\}|=|\{\text{proj.inj.}\Theta\text{-modules}\}|$
Therefore: $\Hom_{\Lambda}(\cP,Q)$ is a projective-injective $\Theta$-module if and only if $Q$ is a filtered projective-injective $\Lambda$-module, as claimed.
\end{proof}

\begin{corollary} \label{cor_new_proj_is_inj}
Let $\Lambda$ be a cyclic Nakayama algebra, and $\Theta:=\bm\varepsilon(\Lambda)$.
If $\Lambda$ is defect invariant, then each non-filtered projective module $P$ is also a minimal projective module.
\end{corollary}

If $\Lambda$ is the defect invariant reverse of given algebra $\Theta$, then Propositions \ref{length_defect} and \ref{length_defect_2} determine the length of each indecomposable projective $\Lambda$-module. Here we present them as corollaries.

\begin{corollary}\label{emre ddd 1 cor}
Let $\Lambda$ and $\Theta$ be Nakayama algebras such that $\bm\varepsilon(\Lambda)\cong\Theta$. If $\defect\Lambda=\defect\Theta$, then $\ell_{\Lambda}(\Delta(S))=\defect_{\Theta}(P(\tau S))+1$.
\end{corollary}

\begin{proof}
By Proposition \ref{length_defect}, $\ell(\Delta(T))\geq \defect_{\Theta}(P(\tau T))+1$ for every simple $\Theta$-module $T$. Without loss of generality, assume that the simple module $T_1$ satisfies
\begin{align*}
\ell(\Delta(T_1))> \defect_{\Theta}(P(\tau T_1))+1,
\end{align*}
i.e., strict inequality. For $2\leq i\leq \rank\Theta$, let
\begin{align*}
\ell(\Delta(T_i))\geq \defect_{\Theta}(P(\tau T_i))+1.
\end{align*}
The sum of all inequalities is
\begin{align*}
\sum\limits^{\rank\Theta}_{i=1}\ell_{\Lambda}(\Delta(T_i))>\sum\limits^{\rank\Theta}_{i=1}(\defect_{\Theta}(P(\tau T_i))+1).
\end{align*}
The left-hand side is $\rank\Lambda$ by Proposition \ref{base-set}. The right-hand side is $\defect\Theta+\rank\Theta$. We get $\rank\Lambda>\rank\Theta+\defect\Theta$. By Proposition \ref{C_and_A}, $\rank\Lambda=\rank\Theta+\defect\Lambda$, which implies $\defect\Lambda>\defect\Theta$, contradicting the assumption in the statement. Hence, any simple $\Theta$-module $T$ satisfies $\ell(\Delta(T))= \defect_{\Theta}(P(\tau T))+1$.
\end{proof}

 \begin{corollary}\label{cor:defect_delta}
 Let $\Lambda$ be a cyclic Nakayama algebra and $\Theta =\bm\epsilon(\Lambda)$ be its syzygy filtered algebra. Then $\defect \Lambda\geq \defect \Theta$. The equality holds if and only if $\defect_\Lambda(I(\Delta T))=\defect_\Theta P(\tau T)$, for each simple $\Theta$-module $T$.
 \end{corollary}

 \begin{proof}
     The first statement follows from Proposition \ref{C_and_A}. If we assume $\defect\Lambda=\defect\Theta$, then by Corollary \ref{emre ddd 1 cor} we get $\ell_{\Lambda}(\Delta(T))=\defect_{\Theta}(P(\tau T))+1$. By Proposition \ref{prop:B_eqv_def}, we conclude that $\defect_{\Lambda}(I(\Delta(T)))=\defect_{\Theta}(P(\tau T))$. 
 \end{proof}

\begin{corollary}\label{lengths of delta proj} 
Let $\Lambda$ be a cyclic Nakayama algebra, and $\Theta:=\bm\varepsilon(\Lambda)$. If $\Lambda$ is defect invariant, then the length of $P(\Delta(S))\in Filt(\cB(\Lambda))$, where $S$ is simple $\Theta$-module, is given by
\begin{align}
\ell_{\Lambda}\left( P (\Delta(S))\right)=\ell_\Theta (P(S))+\sum\limits^{\ell_{\Theta} (P(S)) }_{i=1}  \defect_\Theta(P(\tau^i S)).
\end{align}

\end{corollary}
\begin{proof}
  \begin{align*}
\ell_{\Lambda}\left( P (\Delta(S))\right) &= \sum\limits_{i=0}^{\ell_\Theta (P(S))-1} \ell_{\Lambda} (\Delta(\tau^i S))\\
&=\sum\limits_{i=0}^{\ell_\Theta (P(S))-1}  (\defect_{\Theta}(P(\tau^{i+1} S)+1)\\
&=\sum\limits^{\ell_\Theta (P(S)) }_{i=1}  (\defect_\Theta(P(\tau^i S))+1)\\
&=\ell_\Theta (P(S))+\sum\limits^{\ell_\Theta (P(S)) }_{i=1}  \defect_\Theta(P(\tau^i S))  .
\end{align*}
\end{proof}

\section{Reverse of $\varepsilon$-construction} \label{section reverse}

So far we have assumed the existence of both Nakayama algebras $\Lambda$ and $\Theta$ where $\Theta = \varepsilon(\Lambda)$. In this section we will start with a Nakayama algebra $\Theta$ and construct an admissible sequence of integers that will serve as Kupish series for an algebra $\Lambda$ such that $\varepsilon(\Lambda) = \Theta$.\\\\




\subsection{Defect invariant case}

{An outline of the process of creating admissible sequence of integers for a Kupish series, which will be done properly in the next proposition, is as follows:}

{We start with a Nakayama algebra $\Theta$ and a complete set of representatives of indecomposable projective $\Theta$-modules $\{P_1,\dots,P_n\}$. 
Next we consider a sequence of integers $\cK=\left(c_1,\ldots,c_N\right)$ and show that this sequence is a Kupish series for a Nakayama algebra $\Lambda$ providing that it satisfies four conditions: 1), 2), 3), 4) as stated in the Proposition \ref{emre defect invariant prop}. Before proving this proposition we will explain the role or purpose of each of these conditions (as they are numbered in the proposition)..
\begin{enumerate}
    \item $N=n+\defect\Theta$ is a statement   that $N$, the number of indecomposable projective $\Lambda$-modules has to 
    be larger 
    by $\defect\Theta$, from $n$, which is the number of indecomposable projective $\Theta$-modules. This is necessary  
    according to \ref{C_and_A} and \ref{rank_rel_def}.
    \item $\cN=(c_{z_1},c_{z_2},\ldots,c_{z_n})$ will be the lengths of the filtered projective $\Lambda$-modules. These lengths are described in terms of defects of the original projective $\Theta$-modules.
By corollaries \ref{cor:defect_delta} and \ref{lengths of delta proj}, each element of the base-set of $\mathcal B(\Lambda)$ can be expressed in terms of the defects of projectives of $\Theta$. Consequently the lengths of all filtered  $\Lambda$-projective 
can be described in terms of defects of projective $\Theta$-modules.
    \item This is the statement that the number of non-filtered $\Lambda$-projectives is determined by the defect of $\Theta$-projectives.
    \item This is numerical condition which describes lengths of non-filtered $\Lambda$-projectives. Here proposition \ref{length_defect_2}
is used.
\end{enumerate}
}

\begin{proposition}\label{emre defect invariant prop}
Let $\Theta$ be a Nakayama algebra and $\left\{P_1,\ldots,P_n\right\}$ the complete set of representatives of isomorphism classes of indecomposable projective $\Theta$-modules. Let $\cK=(c_1,\ldots,c_N)$ be a sequence of positive integers which satisfy the following four conditions:
\begin{enumerate}[label=\arabic*)]
\item\label{emre defect invariant prop rank} $N=n+\defect\Theta$.
\item\label{emre defect invariant prop filtered} $\cN=(c_{z_1},c_{z_2},\ldots,c_{z_n})$ is a subsequence of $\cK$  such that
\begin{align}
c_{z_i}=\ell(P_i)+\sum\limits^{\ell (P_i)}_{j=1} (\defect(P_{\left[i+j\right]}).
\end{align}
\item\label{emre defect invariant prop number of nonfiltered} The number of terms in $\cK$ between $c_{z_i}$ and $c_{z_{i+1}}$ is equal to $\defect (P_{i+1})$. The number of terms in $\cK$ between $c_1$ and $c_{z_1}$ is equal to $\ell(P_1)-a_0-2$ where $a_0$ is the number of terms in $\cK$ between $c_{z_n}$ and $c_N$.

\item \label{emre defect invariant prop unfiltered lengths}  Let $z_i<a_1<a_2<\cdots<a_{\defect(P_{i+1})}<z_{i+1}$ be indices of subsequence of $\cK$. Let $c_{a_m}=\ell(P_i)+\sum\limits^{\ell (P_i)+m}_{j=1} \defect(P_{\left[i+j\right]})$ for each $a_m$.

\end{enumerate}
Then, $\cK$ is an admissible sequence for a Kupisch series of a cyclic Nakayama algebra. 
\end{proposition}

To verify a given sequence of positive integers $(c_1,\ldots,c_N)$ forms a Kupisch series of a cyclic Nakayama algebra, we need to check, by Lemma \ref{kupisch how to verify}, that for each $1\leq i\leq N$, either $c_i=c_{i+1}+1$ or $c_i\leq c_{i+1}$. Therefore, to prove Proposition \ref{emre defect invariant prop}, we  compute all possible values of $c_i-c_{i+1}$ for $c_i,c_{c_{i+1}}\in\cK$ in the following lemmas.

\begin{lemma}\label{sunday lemma when difference is one}  Let $c_{z_i}$, $c_{z_{i+1}}$ be consecutive elements of $\cK$. Recall that
\begin{align*}
c_{z_i}=\ell(P_i)+\sum\limits^{\ell (P_i)}_{j=1} \left(\defect(P_{\left[i+j\right]})\right),\quad \
c_{z_{i+1}}=\ell(P_{i+1})+\sum\limits^{\ell (P_{i+1})}_{j=1} \left(\defect(P_{\left[i+j+1\right]})\right).
\end{align*}

Then,
\begin{enumerate}
\item $c_{z_i}-c_{z_{i+1}}=1$ if and only if $\ell(P_i)=1+\ell(P_{i+1})$.
\item  $\ell(P_i)\leq\ell(P_{i+1})$ implies  
  \begin{align*}
 c_{z_i}-c_{z_{i+1}}=-\sum\limits^{\ell(P_{i+1})}_{j=\ell(P_i)}\defect P_{[i+j+1]}=-\sum\limits^{1+\defect(P_{i+1})}_{j=1} \left(\defect(P_{\left[i+\ell(P_i)+j\right]})\right).
 \end{align*}
 \end{enumerate}
\end{lemma}

\begin{proof}
1) We start with $\ell(P_i)=1+\ell(P_{i+1})$. Then,
\begin{align*}
c_{z_i}&=\ell(P_i)+\sum\limits^{\ell (P_i)}_{j=1} \defect(P_{\left[i+j\right]})\\&=1+\ell(P_{[i+1]})+\sum\limits^{1+\ell(P_{i+1})}_{j=1}\defect(P_{[i+j]})
\\&=1+\ell(P_{[i+1]})+\defect(P_{[i+1]})+\sum\limits^{1+\ell(P_{i+1})}_{j=2}\defect(P_{[i+j]})
\\&=1+\ell(P_{[i+1]})+\defect(P_{[i+1]})+\sum\limits^{\ell(P_{i+1})}_{j=1}\defect(P_{[i+j+1]})
\\&=1+\defect(P_{[i+1]})+\ell(P_{[i+1]})+\sum\limits^{\ell(P_{i+1})}_{j=1}\defect(P_{[i+j+1]})
\\&=1+\defect(P_{[i+1]})+c_{z_{i+1}}=1+c_{z_{i+1}}.
\end{align*}
Since $\ell(P_i)=1+\ell(P_{i+1})$, $\rad P_i\cong P_{i+1}$ implies that $\defect(P_{i+1})=0$.

2) Now we assume that $\ell(P_i)\leq \ell(P_{i+1})$. By Remark \ref{kupisch properties}, $\ell(P_{i+1})-\ell(P_i)=\defect(P_{i+1})$. So,
\begin{align*}
c_{z_{i+1}}&=\ell(P_{i+1})+\sum\limits^{\ell (P_{i+1})}_{j=1} \defect(P_{\left[i+1+j\right]})=\sum\limits^{\ell (P_{i+1})}_{j=1} \left(\defect(P_{\left[i+1+j\right]})+1\right)\\
&=\sum\limits^{\ell (P_i)-1}_{j=1} \left(\defect(P_{\left[i+1+j\right]})+1\right)+\sum\limits^{1+\defect(P_{i+1})}_{j=1} \left(\defect(P_{\left[i+\ell(P_i)+j\right]})+1\right)
\end{align*}

On the other hand,
\begin{align*}
c_{z_i}&=\ell(P_i)+\sum\limits^{\ell (P_i)}_{j=1} \defect(P_{\left[i+j\right]})=1+\defect(P_{i+1})+\sum\limits^{\ell (P_i)-1}_{j=1} \left(\defect(P_{\left[i+j+1\right]})+1\right)
\end{align*}
It follows that
\begin{align*}
c_{z_{i}}-c_{z_{i+1}}&=1+\defect(P_{i+1})-\left(\sum\limits^{1+\defect(P_{i+1})}_{j=1} \left(\defect(P_{\left[i+\ell(P_i)+j\right]})+1\right)\right)\\
&=-\sum\limits^{1+\defect(P_{i+1})}_{j=1} \defect(P_{\left[i+\ell(P_i)+j\right]})=-\sum\limits^{\ell(P_{i+1})}_{j=\ell(P_i)}\defect(P_{[i+j+1]})
\end{align*}
since there are $1+\defect(P_{i+1})$-many terms with summand one.
\end{proof}

\begin{lemma}\label{sunday lemma when difference is not one, all three possibilities here} Let $c_{a_m},c_{a_{m+1}}\in\cK-\cN$ such that $a_m+1=a_{m+1}$. Then,
\begin{align*}
c_{a_m}-c_{a_{m+1}}=\defect(P_{[i+\ell(P_i)+m+1]}).
\end{align*}
Let $c_{a_1}\notin\cN$, $c_{z_i}\in\cN$ such that $c_j=c_{z_i}, c_{j+1}=c_{a_1}$. Then,
\begin{align*}
c_{z_i}-c_{a_1}=-\defect (P_{[i+\ell(P_i)+j]}).
\end{align*}
Let $c_{a_m}\notin\cN$, $c_{z_{i+1}}\in\cN$ such that $c_j=c_{a_m}, c_{j+1}=c_{z_{i+1}}$. Then,
\begin{align*}
c_{a_m}-c_{z_{i+1}}=-\defect(P_{[i+\ell(P_{i+1})+1]}).
\end{align*}
\end{lemma}
\begin{proof}
By Proposition \ref{emre defect invariant prop} \ref{emre defect invariant prop unfiltered lengths}, 
\begin{align*}
    c_{a_{m}}&=\ell(P_i)+\sum\limits^{\ell (P_i)+m}_{j=1} \defect(P_{\left[i+j\right]})\\
    &=-m+\sum\limits^{\ell (P_i)}_{j=1} \left(\defect(P_{\left[i+j\right]})+1\right)+\sum\limits^{m}_{j=1}\left( \defect P_{[i+\ell(P_i)+j]}+1\right)
\end{align*}
Similarly, 
\begin{align*}
    c_{a_{m+1}}&=\ell(P_i)+\sum\limits^{\ell (P_i)+m+1}_{j=1} \defect(P_{\left[i+j\right]})\\&=-(m+1)+\sum\limits^{\ell (P_i)}_{j=1} \left(\defect(P_{\left[i+j\right]})+1\right)+\sum\limits^{m+1}_{j=1}\left( \defect P_{[i+\ell(P_i)+j]}+1\right)
\end{align*}
Therefore the difference $c_{a_m}-c_{a_{m+1}}=-m+(m+1)-\defect(P_{[i+\ell(P_i)+m+1]})=\defect(P_{[i+\ell(P_i)+m+1]})$.\\
For the second equality, we can use the following expressions:
\begin{gather*}
c_{z_i}=\sum\limits^{\ell (P_i)}_{j=1} \left(\defect(P_{\left[i+j\right]})+1\right)\\
c_{a_1}=-1+\sum\limits^{\ell (P_i)}_{j=1} \left(\defect(P_{\left[i+j\right]})+1\right)+\left( \defect P_{[i+\ell(P_i)+j]}+1\right).
\end{gather*}
Therefore $c_{z_i}-c_{a_1}=-\defect (P_{[i+\ell(P_i)+j]})$.\\
For the third equality, notice that $m=\defect(P_{i+1})$. We get 
\begin{align*} 
c_{a_m}=-\defect(P_{i+1})+\sum\limits^{\ell (P_i)}_{j=1} \left(\defect(P_{\left[i+j\right]})+1\right)+\sum\limits^{\defect(P_{i+1})}_{j=1}\left( \defect P_{[i+\ell(P_i)+j]}+1\right)
\end{align*}

On the other hand,

\begin{align*}
c_{z_{i+1}}=&\sum\limits^{\ell (P_{i+1})}_{j=1} \left(\defect(P_{\left[i+1+j\right]})+1\right)\\
=&\sum\limits^{\ell (P_i)}_{j=1} \left(\defect(P_{\left[i+1+j\right]})+1\right)+\sum\limits^{\defect(P_{i+1})}_{j=1} \left(\defect(P_{\left[i+\ell(P_i)+j+1\right]})+1\right)
\end{align*}
The difference is
\begin{align*}
-\defect(P_{[i+\ell(P_i)+\defect(P_{i+1})+1]})=-\defect(P_{[i+\ell(P_{i+1})+1]})
\end{align*}

\end{proof}

\begin{proof}[Proof of Proposition \ref{emre defect invariant prop}]
To verify whether a given sequence of integers $\cK=(c_1,\ldots,c_n)$ forms a Kupisch series of a cyclic Nakayama algebra, we need to check that 
\begin{enumerate}[label=\roman*)]
\item each $c_i\geq 2$
\item and if $c_i>c_{i+1}$, then $c_i-c_{i+1}=1$. 
\end{enumerate} If $c_i\leq c_{i+1}$, $c_{i+1}$ can take any positive integer value, so it satisfies admissibility.
By lemmas \ref{sunday lemma when difference is one} and \ref{sunday lemma when difference is not one, all three possibilities here}, the case $c_j>c_{j+1}$ only happens when $c_j=c_{z_i}$ and $c_{j}=c_{z_{i+1}}$ with $\ell(P_i)=1+\ell(P_{i+1})$. It follows that the difference is one, so $\cK$ is admissible sequence.\\
Each $c_j\in\cK$ has at least $\ell(P_{j'})$-many nonzero summands. If $\Theta$ is cyclic Nakayama algebra, then $\ell(P_i)\geq 2$ for any $i$ which makes $c_j\geq 2$. Let $\Theta$ be a linear Nakayama algebra. For simple projective $T_i$, the corresponding $c_i$ is 
$\ell(P(T_{i+1}))+\sum\limits^{\ell(P(T_{i+1}))}_{j=1}\defect(P_{[i+j]})$. Since $\defect P(T_{i+1})\neq 0$, we get $c_i\geq 2$.

\end{proof}

\begin{proposition}\label{sunday proposition regarding algebra A} Let $\Theta$ be a Nakayama algebra of rank $n$. Let $\cK$ be the admissible sequence constructed from $\Theta$ in proposition \ref{emre defect invariant prop}. If a cyclic Nakayama algebra $\Lambda$ is given by admissible sequence $\cK$, then,  $\bm\varepsilon(\Lambda)\cong\Theta$ and $\defect \Lambda=\defect\Theta$.
\end{proposition}

\begin{proof}
In proposition \ref{emre defect invariant prop}, we construct the sequence $\cK$ for any given Nakayama algebra $\Theta$, and proved that it is admissible. Let $\Lambda$ be the algebra given by such $\cK$. Let $\left\{Q_1,\ldots,Q_{N}\right\}$ be the complete set of representatives of isomorphism classes of indecomposable $\Lambda$-modules and $\ell(Q_i)=c_i\in\cK$. We divide the proof into the following steps.

\begin{enumerate}
    \item $\vert\cS(\Lambda)\vert=n$. 
    \item $\defect \Lambda=\defect\Theta$.
    \item $\cB(\Lambda)$-filtered indecomposable projective $\Lambda$-modules are $Q_{z_1},\ldots,Q_{z_n}$.
\end{enumerate}
1) By lemma \ref{sunday lemma lengths of delta and cardinality}, $\vert\cS(\Lambda)\vert=\vert\{c_i\mid c_i\leq c_{i+1}, 1\leq i\leq N\}$ where $c_{N+1}=c_1$. There are two types of indecomposable projective $\Lambda$-modules, either $\ell(Q_i)=c_{i}\notin\cN$ or $\ell(Q_i)=c_i\in\cN$ where $\cN\subset\cK$ by the construction given in proposition \ref{emre defect invariant prop}. We have the following subsets
\begin{align*}
\{c_i\mid c_i\leq c_{i+1},\,1\leq i\leq N\}=\{c_j\mid c_j\leq c_{j+1},\,c_j\notin\cN\}\cup\{c_{z_{i}}\mid  c_{z_i}\leq c_{z_{i+1}},\,c_{z_i}\in\cN\}.
\end{align*}

By modifying lemma \ref{sunday lemma when difference is not one, all three possibilities here} for new indices, any $c_j\notin\cN$ satisfies $c_j\leq c_{j+1}$. Therefore the cardinality of the first set is  $\vert\cK\vert-\vert\cN\vert=\defect\Theta$.

Now, we consider the partition of the set $\{c_{z_i}\mid c_{z_i}\in\cN\}$ of size $n=\vert\cN\vert$  into
\begin{align*}
\{c_{z_i}\mid c_{z_i}\leq c_{z_{i+1}},\,c_{z_{i}}\in\cN\}\cup \{c_{z_{i+1}}\mid c_{z_i}> c_{z_{i+1}},\,c_{z_{i}}\in\cN\}.
\end{align*}
We claim that the cardinality of the first subset is $n-\defect\Theta$. Because, by lemma \ref{sunday lemma when difference is one}, $c_{z_i}>c_{z_{i+1}}$ if and only if $\ell(P_i)=1+\ell(P_{i+1})$. By lemma \ref{sunday lemma lengths of delta and cardinality}, $\vert\{\ell(P_{i+1})\,\vert\, \ell(P_i)>\ell(P_{i+1}) \}\vert=\defect\Theta$ which makes the size of the first set $n-\defect\Theta$. Therefore
\begin{align*}
\vert\cS(\Lambda)\vert&=\vert\{c_j\mid c_j\leq c_{j+1},\,j\notin\cN\}\vert+\vert\{c_{z_i}\mid c_{z_i}\leq c_{z_{i+1}},\,c_{z_{i}}\in\cN\}\vert\\&=\defect\Theta+\left(n-\defect\Theta\right)=n. 
\end{align*}\\
2) By Poposition \ref{rank_rel_def}, $\rank \Lambda=\defect \Lambda+\#\rel \Lambda$. By the previous step, $\vert\cS(\Lambda)\vert=n=\#\rel $. Since $\rank \Lambda=\vert\cK\vert=n+\defect \Theta$, claim follows.\\
3) Notice that there are $n$ $\cB(\Lambda)$-filtered indecomposable projective $\Lambda$-modules, since $\vert\cB(\Lambda)\vert=\vert\cS(\Lambda)\vert=n$. 
The lengths of elements of $\cB(\Lambda)$ are $c_{i+1}-c_i+1$ provided that $c_i\leq c_{i+1}$ by lemma \ref{sunday lemma lengths of delta and cardinality}. In lemmas \ref{sunday lemma when difference is one} and \ref{sunday lemma when difference is not one, all three possibilities here}, we calculated all possible differences of $c_i-c_{i+1}$ which are always of the form $-\defect(P_j)$ when $c_i\leq c_{i+1}$. Moreover, by proposition \ref{emre defect invariant prop} \ref{emre defect invariant prop filtered}, for any $1\leq i\leq n$, $\defect(P_i)+1$ is a summand of $Q_{z_i}$. Therefore, every $Q_{z_i}$ is $\cB(\Lambda)$-filtered. Because there are $n$-many $Q_{z_i}$, hence we get all.  \\

Now we can show that $\bm\varepsilon(\Lambda)\cong\Theta$. It is enough to compute dimensions of projective $\bm\varepsilon(\Lambda)$-modules, since indecomposable $\bm\varepsilon(\Lambda)$-modules are uniserial, i.e. dimension and length of indecomposable modules are equal. Therefore, $\dim_{\mathbb{K}}\Hom_A\left(\bigoplus\limits_{1\leq i\leq n} Q_{z_i}, Q_{z_i}\right)=\ell_{\cB(\Lambda)}(Q_{z_i})=\ell(P_i)$, since the number of elements $\cB(\Lambda)$ appearing in each $Q_{z_i}$ is $\ell(P_i)$. Hence Kupisch series of $\bm\varepsilon(\Lambda)$ is equal to Kupisch series of $\Theta$ upto cyclic permutation.
\end{proof}

\subsection{General case, i.e. Not Defect-Invariant}
Let $\Theta$ be a Nakayama algebra of rank $n$, and let $\{P_1,P_2,\ldots,P_n\}$ be the ordered set of indecomposable projective $\Theta$-modules.
In this subsection, we will construct cyclic Nakayama algebra $\Lambda$ such that $\varepsilon(\Lambda)\cong \Theta$ and $\defect(\Lambda)>\defect(\Theta)$.


\begin{definition} For any finite sequence of integers $(c_1,\ldots,c_n)$ we define the distance between $c_i$ and $c_j$ with $j>i$ as $\min \{[j-i+1],[i-j+1]\}$, denoted by $\dist(c_i,c_j)$.
\end{definition}

{We start with a Nakayama algebra $\Theta$ and a complete set of representatives of indecomposable projective $\Theta$-modules $\{P_1,\dots,P_n\}$. 
Next we consider a sequence of integers $\cK=\left(c_1,\ldots,c_N\right)$ and show that this sequence is a Kupish series for a Nakayama algebra $\Lambda$ providing that it satisfies four conditions: 1), 2), 3), 4) as stated in the Proposition \ref{emre defect invariant prop}. Before proving this proposition we will explain the role or purpose of each of these conditions (as they are numbered in the proposition)..

Now we present the technical proposition, which enables us to construct a cyclic Nakayama algebra for a given connected Nakayama algebra $\Theta$ and a vector $w=(w_1,\dots,w_n)$ such that $w_i\geq \defect(P_i)$. By \ref{rank_rel_def}, the rank of $\Lambda$ is $\rank\Theta+\sum_{i}w_i$ which is at part \ref{emre general rank}. Each $w_i+1$ will correspond to length of filtered modules which is hidden in part \ref{emre general filtered}. For non-filtered $\Lambda$ modules, by \ref{length_defect_2}, there are $\defect P_{i+1}$-many possible socles and $w_{i+1}$ many nonisomorphic tops, which we describe the lengths in part \ref{emre general length of unfiltered}. There are $w_i$ non-filtered projective modules between filtered projectives which is in part \ref{emre general number of unfiltered}. Below is the proposition. 

\begin{proposition}\label{emre general proposition non defect invariant} Let $\Theta$ be a Nakayama algebra $\{P_1,\ldots,P_n\}$ be the complete set of representatives of isomorphism classes of indecomposable projective $\Theta$-modules. Let $w\in\mathbb{R}^n$ be a vector satisfying $w_i\geq \defect(P_i)$. Let $\cK=(c_1,\ldots,c_N)$ be a sequence of positive integers which satisfy the following  conditions:

\begin{enumerate}[label=\arabic*)]
\item\label{emre general rank} $N=n+\sum^n_{i=1} w_i$ 
\item\label{emre general filtered} $\cN=(c_{z_1},c_{z_2},\ldots,c_{z_n})$ is a subsequence of $\cK$ such that

\begin{align}
c_{z_i}=\sum\limits^{\ell (P_i)}_{j=1} \left(w_{\left[i+j\right]}+1\right)
\end{align}

\item\label{emre general number of unfiltered}  {The number of terms in $\cK$ between $c_{z_i}$ and $c_{z_{i+1}}$ is equal to $w_{i+1}$. The number of terms in $\cK$ between $c_1$ and $c_{z_1}$ is equal to $w_1-a_0-2$ where $a_0$ is the number of terms in $\cK$ between $c_{z_n}$ and $c_N$.}

\item\label{emre general length of unfiltered} For any term $c_{a_j} \in\cK-\cN$ which satisfies ${z_i}<{a_1}<\cdots<{a_{w_{i+1}}}<{z_{i+1}}$, there is a subsequence of indices  $\{b_1, b_2, \dots , b_{\defect(P_{i+1})}\}\subset \{a_1,\ a_2,\dots , a_{w_{i+1}}\}$ such that

\begin{align} 
c_{b_m}=-\dist(c_{b_m},c_{z_{i}})+\sum\limits^{\ell (P_i)}_{j=1} \left(w_{\left[i+j\right]}+1\right)+\sum\limits^{m}_{j=1}\left( w_{[i+\ell(P_i)+j]}+1\right)
\end{align} 
Any term $c_{b_m} \in\cK-\cN$ with index either $z_n<b_m\leq N$ or $1\leq b_m< z_1$   is given by 
\begin{align} 
c_{b_m}=-\dist(c_{b_m},c_{z_n})+\sum\limits^{\ell (P_n)}_{j=1} \left(w_{\left[j\right]}+1\right)+\sum\limits^{m}_{j=1}\left( w_{[\ell(P_n)+j]}+1\right)
\end{align} 
\item For each $b_t< a_m < b_{t+1}$, we set 

\begin{align}
c_{a_m}=-\dist(c_{a_m},c_{z_{i}})+\sum\limits^{\ell (P_i)}_{j=1} \left(w_{\left[i+j\right]}+1\right)+\sum\limits^{t}_{j=1}\left( w_{[i+\ell(P_i)+j]}+1\right)
\end{align}

\item  For each $z_i< a_m < b_{1}$, we set

\begin{align}
c_{a_m}=-\dist(c_{a_m},c_{z_{i}})+\sum\limits^{\ell (P_i)}_{j=1} \left(w_{\left[i+j\right]}+1\right)
\end{align}

\item  For each $b_{\defect P_{i+1}}< a_m < z_{i+1}$, we set
\begin{align}
c_{a_m}=-\dist(c_{a_m},c_{z_{i}})+\sum\limits^{\ell (P_i)}_{j=1} \left(w_{\left[i+j\right]}+1\right)+\sum\limits^{\defect P_{i+1}}_{j=1}\left( w_{[i+\ell(P_i)+j]}+1\right)
\end{align}

\end{enumerate}
The sequence $\cK$ satisfying all these conditions forms an admissible sequence of a cyclic Nakayama algebra. 
\end{proposition}

    To verify whether a given sequence of integers $\cK=(c_1,\ldots,c_n)$ forms a Kupisch series of a cyclic Nakayama algebra, we need to check that 
\begin{enumerate}[label=\roman*)]
\item each $c_i\geq 2$
\item and if $c_i>c_{i+1}$, then $c_i-c_{i+1}=1$. 
\end{enumerate} If $c_i\leq c_{i+1}$, $c_{i+1}$ can take any positive integer value, so it satisfies admissibility. Hence, to prove proposition \ref{emre general proposition non defect invariant}, we need to compute all possible $c_i-c_{i+1}$ for $c_i,c_{c_{i+1}}\in\cK$ which are stated in the following two lemmas \ref{emre general lemma when difference is one} and \ref{emre general lemma when difference is not one, all three possibilities here}.

\begin{lemma}\label{emre general lemma when difference is one} Let $c_{z_i}$ and $c_{z_{i+1}}$ be consecutive elements of $\cK$. Then,
\begin{enumerate}
    \item $c_{z_i}-c_{z_{i+1}}=1$ if and only if $\ell(P_i)=1+\ell(P_{i+1})$ 
    \item $c_{z_i}-c_{z_{i+1}}<0$ if and only if $\ell(P_i)=\ell(P_{i+1})$.
\end{enumerate}
In the latter case, we have
\begin{align}
    c_{z_i}-c_{z_{i+1}}=-w_{[i+\ell(P_i)+1]}.
    \end{align}
\end{lemma}

\begin{proof} Notice that $c_{z_i}$ and $c_{z_{i+1}}$ are consecutive elements of $\cK$ if and only if $w_{i+1}=0$ by construction of $\cK$ in proposition \ref{emre general proposition non defect invariant} \ref{emre general number of unfiltered}. \\

We start with $\ell(P_i)=1+\ell(P_{i+1})$. Then,
\begin{align*}
c_{z_i}&=\sum\limits^{\ell (P_i)}_{j=1} \left(w_{\left[i+j\right]}+1\right)\\
&=1+(w_{i+1})+\sum\limits^{\ell (P_{i+1})+1}_{j=2} \left((w_{\left[i+j\right]})+1\right)\\
&=1+(w_{i+1})+\sum\limits^{\ell (P_{i+1})}_{j=1} \left((w_{\left[i+j+1\right]})+1\right)\\
&=1+(w_{i+1})+c_{z_{i+1}}=1+c_{z_{i+1}}.
\end{align*}

The difference $c_{z_i}-c_{z_{i+1}}$ is
\begin{align}
    c_{z_i}-c_{z_{i+1}}&=\sum\limits^{\ell (P_i)}_{j=1} \left(w_{\left[i+j\right]}+1\right)-\sum\limits^{\ell (P_{i+1})}_{j=1} \left(w_{\left[i+j+1\right]}+1\right)\\
    &=1+w_{[i+1]}-\sum\limits^{\ell(P_{i+1})} _{j=\ell(P_i)}(w_{[i+j+1]}+1)\\
     &=1-\sum\limits^{\ell(P_{i+1})} _{j=\ell(P_i)}(w_{[i+j+1]}+1)
\end{align}
Therefore, $c_{z_i}-c_{z_{i+1}}=1$ implies that the second summand is zero. This is possible when $\ell(P_i)+i=\ell(P_{i+1})+i+1$, so $\ell(P_i)=1+\ell(P_{i+1})$.
\end{proof}
Now we assume that $\ell(P_i)\leq \ell(P_{i+1})$. By remark \ref{kupisch properties}, $\ell(P_{i+1})-\ell(P_i)=\defect(P_{i+1})$. So,
\begin{align*}
c_{z_{i+1}}&=\sum\limits^{\ell (P_{i+1})}_{j=1} \left((w_{\left[i+1+j\right]})+1\right)\\
&=\sum\limits^{\ell (P_i)-1}_{j=1} \left((w_{\left[i+1+j\right]})+1\right)+\sum\limits^{1+\defect(P_{i+1})}_{j=1} \left((w_{\left[i+\ell(P_i)+j\right]})+1\right)
\end{align*}

On the other hand,
\begin{align*}
c_{z_i}&=\sum\limits^{\ell (P_i)}_{j=1} \left((w_{\left[i+j\right]})+1\right)=1+(w_{i+1})+\sum\limits^{\ell (P_i)-1}_{j=1} \left((w_{\left[i+j+1\right]})+1\right)
\end{align*}
It follows that
\begin{align*}
c_{z_{i}}-c_{z_{i+1}}&=1+(w_{i+1})-\left(\sum\limits^{1+\defect(P_{i+1})}_{j=1} \left((w_{\left[i+\ell(P_i)+j\right]})+1\right)\right)\\
&=1-\sum\limits^{1+\defect(P_{i+1})}_{j=1} (w_{\left[i+\ell(P_i)+j\right]}+1)=-w_{[i+\ell(P_i)+1]}\leq 0
\end{align*}
since $w_{i+1}=0$ and this implies that $\defect P_{i+1}=0$. Hence we get $c_{z_i}-c_{z_{i+1}}\leq 0$.

\begin{lemma}\label{emre general lemma when difference is not one, all three possibilities here}
Let $c_{a_m},c_{a_{m+1}}\in\cN$ such that $a_m+1=a_{m+1}$. Then, 

\begin{enumerate}
    \item if there exist indices $b_t<a_m<a_{m+1}<b_t$, then $c_{a_m}-c_{a_{m+1}}=1$
    \item if there exist indices $z_{i}<a_m<a_{m+1}<b_1$, then $c_{a_m}-c_{a_{m+1}}=1$
    \item if there exist indices $b_t=a_m<a_{m+1}<b_{t+1}$, then $c_{a_m}-c_{a_{m+1}}=1$
    \item if $a_m=b_t$, $a_{m+1}=b_{t+1}$, then $c_{a_m}-c_{a_{m+1}}=-w_{[i+\ell(P_i)+t+1]}$
    \item if $a_m<a_{m+1}=b_{t+1}$, then $c_{a_m}-c_{a_{m+1}}=--w_{[i+\ell(P_i)+t+1]}$.
\end{enumerate}
\end{lemma}

\begin{proof}
    For the first three cases, the difference $c_{a_m}-c_{a_{m+1}}$ is
    \begin{align}
        -\dist(c_{a_m},c_{z_i})+\dist(c_{a_{m+1}},c_{z_i})=        a_{m+1}-a_{m}=1.
        \end{align}
Let $c_{b_t}$, $c_{b_{t+1}}$ be consecutive in $\cK$, i.e., $b_{t+1}=b_t+1$. Then, 
\begin{gather*}
    c_{b_t}-c_{b_{t+1}}=\\
   -\dist(c_{b_t},c_{z_{i}})+\sum\limits^{\ell (P_i)}_{j=1} \left(w_{\left[i+j\right]}+1\right)+\sum\limits^{t}_{j=1}\left( w_{[i+\ell(P_i)+j]}+1\right)\\+\dist(c_{b_{t+1}},c_{z_{i}})-\sum\limits^{\ell (P_i)}_{j=1} \left(w_{\left[i+j\right]}+1\right)-\sum\limits^{t+1}_{j=1}\left( w_{[i+\ell(P_i)+j]}+1\right)
  \\    =-w_{[i+\ell(P_i)+t+1]}
\end{gather*}
Let $c_{a_m}$ and $c_{b_{t+1}}$ be consecutive in $\cK$, i.e., $b_{t+1}=a_{m+1}$. Then,
\begin{gather*}
    c_{a_m}=-\dist(c_{a_m},c_{z_{i}})+\sum\limits^{\ell (P_i)}_{j=1} \left(w_{\left[i+j\right]}+1\right)+\sum\limits^{t}_{j=1}\left( w_{[i+\ell(P_i)+j]}+1\right)\\
    c_{b_{t+1}}=-\dist(c_{b_{t+1}},c_{z_{i}})+\sum\limits^{\ell (P_i)}_{j=1} \left(w_{\left[i+j\right]}+1\right)+\sum\limits^{t+1}_{j=1}\left( w_{[i+\ell(P_i)+j]}+1\right)
\end{gather*}
the difference is $-w_{[i+\ell(P_i)+t+1]}$
\end{proof}

Now, we state the proof of proposition \ref{emre general proposition non defect invariant}.
\begin{proof}
  We want to show that $\cK$ is an admissible sequence of a cyclic Nakayama algebra. Notice that each $c_{i}$ has at least $\ell(P_j)$ many summands, so $c_i\geq 2$ for all $i$. In lemmas \ref{emre general lemma when difference is not one, all three possibilities here} and \ref{emre general lemma when difference is one}, we calculated all possible differences, and show that $c_i>c_{i+1}$ implies that $c_i=c_{i+1}+1$. Hence, $\cK$ is an admissible sequence.
\end{proof}

\begin{proposition}\label{emre general prop epsilon}
    Let $\Theta$ be either a cyclic Nakayama algebra or a connected linear Nakayama algebra of rank $n$. Let $\cK$ be the admissible sequence constructed from $\Theta$ as in proposition \ref{emre general proposition non defect invariant}. If a cyclic Nakayama algebra $\Lambda$ is given by admissible sequence $\cK$, then,  $\bm\varepsilon(\Lambda)\cong\Theta$ and $\defect \Lambda=\vert w\vert\geq\defect\Theta$.
\end{proposition}

\begin{proof}

In proposition \ref{emre defect invariant prop}, we construct the sequence $\cK$ for any given cyclic Nakayama algebra $\Theta$, and proved that it is admissible. Let $\Lambda$ be the algebra given by such $\cK$. Let $\left\{Q_1,\ldots,Q_{N}\right\}$ be the complete set of representatives of isomorphism classes of indecomposable $\Lambda$-modules and $\ell(Q_i)=c_i\in\cK$. We divide the proof into the following steps.
\begin{enumerate}
    \item $\vert\cS(\Lambda)\vert=n$
    \item $\defect \Lambda=\vert w\vert$
    \item $\cB(\Lambda)$-filtered indecomposable projective $\Lambda$-modules are $Q_{z_1},\ldots,Q_{z_n}$.
\end{enumerate}

First, we show that the cardinality of the socle set of $\Lambda$ is $n$, i.e., $\vert\cS(\Lambda)\vert=n$. By lemma \ref{sunday lemma lengths of delta and cardinality}, $\vert\cS(\Lambda)\vert=\vert\{c_i\mid c_i\leq c_{i+1}, 1\leq i\leq N\}$ where $c_{N+1}=c_1$. There are two types of indecomposable projective $\Lambda$-modules, either $\ell(Q_i)=c_{i}\notin\cN$ or $\ell(Q_i)=c_i\in\cN$ where $\cN\subset\cK$ by the construction given in proposition \ref{emre general proposition non defect invariant}. We have the following subsets
\begin{align*}
\{c_i\mid c_i\leq c_{i+1},\,1\leq i\leq N\}=\{c_j\mid c_j\leq c_{j+1},\,c_j\notin\cN\}\cup\{c_{z_{i}}\mid  c_{z_i}\leq c_{z_{i+1}},\,c_{z_i}\in\cN\}.
\end{align*}

The cardinality of the first subset is $\defect\Theta$. Because, by lemma \ref{emre general lemma when difference is not one, all three possibilities here}, $c_j\leq c_{j+1}$ if $c_{j+1}$ is of the form $c_{b_t}$ for some $z_i<b_t<z_{i+1}$. In total, there are $\sum_i\defect(P_i)$-many such terms. Hence the cardinality is $\defect\Theta$.

Now, we consider the partition of the set $\{c_{z_i}\mid c_{z_i}\in\cN\}$ of size $n=\vert\cN\vert$  into
\begin{align*}
\{c_{z_i}\mid c_{z_i}\leq c_{z_{i+1}},\,c_{z_{i}}\in\cN\}\cup \{c_{z_{i+1}}\mid c_{z_i}> c_{z_{i+1}},\,c_{z_{i}}\in\cN\}.
\end{align*}
We claim that the cardinality of the first subset is $n-\defect\Theta$. Because, by lemma \ref{emre general lemma when difference is one}, $c_{z_i}>c_{z_{i+1}}$ if and only if $\ell(P_i)=1+\ell(P_{i+1})$. By lemma \ref{sunday lemma lengths of delta and cardinality}, $\vert\{\ell(P_{i+1})\,\vert\, \ell(P_i)>\ell(P_{i+1}) \}\vert=\defect\Theta$ which makes the size of the first set $n-\defect\Theta$. Therefore
\begin{align*}
\vert\cS(\Lambda)\vert&=\vert\{c_j\mid c_j\leq c_{j+1},\,j\notin\cN\}\vert+\vert\{c_{z_i}\mid c_{z_i}\leq c_{z_{i+1}},\,c_{z_{i}}\in\cN\}\vert\\&=\defect\Theta+\left(n-\defect\Theta\right)=n. 
\end{align*}

Secondly, we compute the defect of $\Lambda$, i.e., $\defect \Lambda=\vert w\vert$. By proposition \ref{rank_rel_def}, $\rank \Lambda=\defect \Lambda+\#\rel \Lambda$. 
Together with the previous claim and proposition \ref{C_and_A}, $\vert\cS(\Lambda)\vert=n=\#\rel\Lambda $. Since $\rank \Lambda=\vert\cK\vert=n+\vert w\vert$, claim follows.

Finally, We show that $\cB(\Lambda)$-filtered indecomposable projective $\Lambda$-modules are $Q_{z_1},\ldots,Q_{z_n}$. Notice that there are $n$ $\cB(\Lambda)$-filtered indecomposable projective $\Lambda$-modules, since $\vert\cB(\Lambda)\vert=\vert\cS(\Lambda)\vert=n$. 
The lengths of elements of $\cB(\Lambda)$ are $c_{i+1}-c_i+1$ provided that $c_i\leq c_{i+1}$ by lemma \ref{sunday lemma lengths of delta and cardinality}. In lemmas \ref{emre general lemma when difference is one} and \ref{emre general lemma when difference is not one, all three possibilities here}, we calculated all possible differences of $c_i-c_{i+1}$ which are always of the form $-w_{[i]}$ when $c_i\leq c_{i+1}$. Moreover, by proposition \ref{emre general proposition non defect invariant} \ref{emre general filtered}, for any $1\leq i\leq n$, $w_{[i]}+1$ is a summand of $Q_{z_i}$. Therefore, every $Q_{z_i}$ is $\cB(\Lambda)$-filtered. Because there are $n$-many $Q_{z_i}$, hence we get all.

Now we can show that $\bm\varepsilon(\Lambda)\cong\Theta$. It is enough to compute dimensions of projective $\bm\varepsilon(\Lambda)$-modules, since indecomposable $\bm\varepsilon(\Lambda)$-modules are uniserial, i.e., dimension and length of indecomposable modules are equal. Therefore,

\begin{align*}
    \dim_{\mathbb{K}}\Hom_{\Lambda}\left(\bigoplus\limits_{1\leq i\leq n} Q_{z_i}, Q_{z_i}\right)=\ell_{\cB(\Lambda)}(Q_{z_i})=\ell(P_i)
    \end{align*}
    since the number of elements $\cB(\Lambda)$ appearing in each $Q_{z_i}$ is $\ell(P_i)$. Hence Kupisch series of $\bm\varepsilon(\Lambda)$ is equal to Kupisch series of $\Theta$ upto cyclic permutation.    
\end{proof}

\begin{example}
    Let $\Theta$ be given by Kupisch series $(2,2,1,3,2,1)$. Then, defect invariant reverse $\Lambda$ is given by $(2,4,3,3,3,4,3,2,2)$. If we choose $w=(1,1,0,3,0,0)$, then 
    \begin{align*}
        (3,\underline{2},4,3,\underline{3},\underline{3},4,3,2,\underline{4})\\
        (3,\underline{2},4,3,\underline{3},\underline{3},4,3,2,\underline{3})\\
        (3,\underline{5},4,3,\underline{3},\underline{3},4,3,2,\underline{4})\\
        (3,\underline{5},4,3,\underline{3},\underline{3},4,3,2,\underline{3})
    \end{align*}
    are Kupisch series of nonisomorphic Nakayama algebras $\Lambda_i$ such that $\bm\varepsilon(\Lambda_i)\cong\Theta$ where $\underline{c_i}$ is the length of non-filtered projective $\Lambda_i$-modules.
\end{example}

\section{Nakayama Algebras with higher homological dimensions}\label{section higher homological}

In this section, our goal is to establish results concerning specific homological dimensions of Nakayama algebras. Initially, we revisit some statements proved in \cite{sen21}. Afterward, we analyze algebras characterized by dominant dimensions in the subsequent subsections. To clarify what we mean by algebras characterized by dominant dimension, we introduce a comprehensive definition here:

\begin{definition}\label{emre all definitions together}
Let $A$ be a finite dimensional algebra over field $\mathbb{K}$. 
   \begin{enumerate}
       \item     $A$ is called \emph{dominant Auslander-Gorenstein} if injective dimension of each indecomposable projective $A$-module is bounded by its dominant dimension \cite{CIM}.
       \item     $A$ is called \emph{dominant Auslander-regular} if $A$ is dominant Auslander-Gorenstein algebra of finite global dimension \cite{CIM}.
       \item     $A$ is called \emph{minimal Auslander-Gorenstein} algebra, if it is dominant Auslander-Gorenstein and nonzero injective dimensions of projective modules are equal \cite{iyamasolberg}.
       \item    $A$ is called \emph{higher Auslander algebra} if it is minimal Auslander-Gorenstein algebra with finite global dimension \cite{iyama}.
       \end{enumerate}
\end{definition}

This section is organized as follows. First, we review some statements from \cite{sen19} and \cite{sen21} regarding injective objects. Next, we present proofs for Theorems \ref{emre thm domdim} and \ref{emre thm minimal aus. gor}. In the subsequent sections, we systematically study algebras as defined in \ref{emre all definitions together}.

\subsection{Properties of Injective Modules and Homological Dimensions}

\begin{lemma}\label{LEM the submodule of delta gives injective}\label{injectivity} Let $B$ be an element of $\cB(\Lambda)$ such that it is submodule of an indecomposable projective-injective $\Lambda$-module $P$ and it has a proper submodule $X$, i.e.,
\begin{align*}
X\hookrightarrow B\hookrightarrow P
\end{align*}
Then the quotient $P/X$ is an injective $\Lambda$-module.
\end{lemma}
\begin{lemma}\label{LEM injective first szyg is submodule of delta}\label{every second syzy of injective} If $I$ is an indecomposable injective non-projective $\Lambda$-module, then $\Omega^1(I)$ is {a proper} submodule of an element  $B$ of the base-set $\cB(\Lambda)$.
\end{lemma}

\begin{proposition}\label{PROP every injective eps mod is second syzygy of injective lambda}\label{GorensteinReduction} If $I$ is an indecomposable injective non-projective $\bm\varepsilon(\Lambda)$-module, then there exists an injective non-projective $\Lambda$-module $J$ such that $\ho{\Omega^2(J)}\cong I$.
\end{proposition}

\begin{proposition}\label{PROP domdim geq 3 implies bijection} Let $\Lambda$ be a cyclic Nakayama algebra such that $\domdim\Lambda\geq 3$. Then, there is a bijection between indecomposable injective {non-projective $\Lambda$-modules $I$ and indecomposable injective non-projective $\bm\varepsilon(\Lambda)$-modules $J$ via the correspondence $I=\Sigma^2(\Delta(J))$ and $\Delta(J)=\Omega^2(I)$.}
\end{proposition}
\begin{corollary} \label{cor of PROP domdim geq 3 implies bijection }
    Let $\Lambda$ be a cyclic Nakayama algebra such that $\domdim\Lambda\geq 3$. Then, $\defect\Lambda=\defect\bm\varepsilon(\Lambda)$.
\end{corollary}

\begin{proposition}\label{different injectives} If $I_1,I_2,\ldots,I_m$ are indecomposable nonisomorphic injective $\Lambda$-modules satisfying $\Omega^2(I_i)\cong \Omega^2(I_{j})$ for all $1\leq i\neq j\leq m$, then projective covers $P(I_1),$ $P(I_2),\ldots,P(I_m)$ are isomorphic. In particular $\topp I_i\cong \topp I_{j}$ and $\topp P(\Omega^1(I_i))\ncong \topp P(\Omega^1(I_{j}))$  for all $1\leq i\neq j\leq m$.
\end{proposition}

\begin{definition}\cite{aus}\label{def domdim} Let $A$ be an artin algebra. The dominant dimension of an $A$-module $M$, 
which we denote by $\domdim_A M$, is the maximum integer $t$ (or $\infty$) having the property that if
$0\rightarrow M\rightarrow I_0\rightarrow I_1\rightarrow\cdots\rightarrow I_t\rightarrow\cdots$ is a minimal injective resolution of $M$, then $I_j$ is projective for all $j<t$ (or $\infty$). Dominant dimension of $A$ is
\begin{align}\label{gercek tanim2}
\domdim A=\inf\left\{\domdim P\,\vert\, P \text{ is projective non-injective } A\, \text{module}\right\}.
\end{align}
Because for any artin algebra $\Lambda$ $\domdim\Lambda=\domdim\Lambda^{op}$ holds {\cite{Mu}},  \ref{gercek tanim2}
 is equivalent to 
 \begin{align}\label{gercek tanim}
\domdim A=\inf\left\{\codomdim I\,\vert\, I \text{ is injective non-projective } A\, \text{module}\right\}.
\end{align}
where $\codomdim_A I$ is the maximum integer $t$ (or $\infty$) having the property that if
$\cdots\rightarrow P_t\rightarrow\cdots\rightarrow P_0\rightarrow I\rightarrow 0$ is a minimal projective resolution of injective module $I$, then $P_j$ is injective for all $j<t$ (or $\infty$).  
Global dimension of algebra $A$ is the supremum of projective dimensions of all finitely generated $A$-modules:
\begin{align*}
\gldim A=\sup\left\{\pdim M\vert\,M\in \moddd A \right\}
\end{align*} where mod-$A$ denotes the category of finitely generated $A$-modules.
We call an algebra Gorenstein if the projective dimensions of injective modules are finite.
\end{definition}
    Proposition \ref{diagram} gives a connection between projective dimensions of $\Lambda$ and $\bm\varepsilon(\Lambda)$ modules in the following way. 
\begin{proposition}\cite{sen19}
    If X is an indecomposable $\cB(\Lambda)$-filtered $\Lambda$-module which is not
projective-injective, then there exist a module $M$ such that $X\cong\Omega^2(M)$. Furthermore,
$M$ can be chosen so that, in its minimal projective presentation $P_1\rightarrow P_0\rightarrow M\rightarrow 0$
both $P_0$ and $P_1$ are projective-injective. In particular $X\rightarrow P_0\rightarrow P_1$ is the injective
copresentation of $X$ and $M=\Omega^{-2}X$.
\end{proposition}
\begin{remark}
    The evaluation functor $\Hom_{\Lambda}(\cP,-)$ is exact, because $\cP$ is a projective module. If $M\in\ Filt(\cB(\Lambda))$ has projective resolution
    \begin{align*}
        \cdots\rightarrow P_2\rightarrow P_1\rightarrow P_0\rightarrow M\rightarrow 0
    \end{align*}
then the projective resolution of $X=\Hom_{\Lambda}(\cP,M)\in\moddd\bm\varepsilon(\Lambda)$ is given by:
  \begin{align*}
        \cdots\rightarrow\Hom_{\Lambda}(\cP, P_2)\rightarrow\Hom_{\Lambda}(\cP, P_1)\rightarrow\Hom_{\Lambda}(\cP, P_0)\rightarrow X\rightarrow 0
    \end{align*}
Each $\Hom_{\Lambda}(\cP,P_i)$, $i\geq 0$ is a projective $\bm\varepsilon(\Lambda)$-module. Any nonsplit exact sequence
\begin{align*}
    0\rightarrow X_2\rightarrow X_1\rightarrow X_0\rightarrow 0
\end{align*}
in $\moddd\bm\varepsilon(\Lambda)$ can be carried into $\moddd\Lambda$ by the categorical equivalence, i.e.,
\begin{align*}
    0\rightarrow \Delta X_2\rightarrow\Delta X_1\rightarrow\Delta X_0\rightarrow 0
\end{align*}
is exact in $\moddd\Lambda$ where each $X_i\cong\Hom_{\Lambda}(\cP,\Delta X_i)$. 
\end{remark}

\begin{remark} \label{emre remark reduction}
Let $\Lambda$ be a cyclic Nakayama algebra. We have the following reductions:

\begin{enumerate}[label=\arabic*)]
\item $\gldim\Lambda=\gldim{\bm\varepsilon}(\Lambda)+2$ if $\gldim\Lambda\geq 2$ \cite{sen19}
\item\label{reduction domdim} $\domdim\Lambda=\domdim{\bm\varepsilon}(\Lambda)+2$ if $\domdim\Lambda\geq 3$ \cite{sen19}
\item $\findim\Lambda=\findim{\bm\varepsilon}(\Lambda)+2$ if $\varphi$-$\dim\Lambda\geq 2$ \cite{sen19}
\item $\del\Lambda=\del{\bm\varepsilon}(\Lambda)+2$ if $\varphi$-$\dim\Lambda\geq 2$ \cite{sen2021delooping}.
\item $\varphi$-$\dim\Lambda=\varphi$-$\dim{\bm\varepsilon}(\Lambda)+2$ if $\varphi$-$\dim\Lambda\geq 2$ \cite{sen18}
\end{enumerate}
where $\findim$, $\del$, $\varphi$-$\dim$ stand for finitistic dimension, delooping level and Igusa-Todorov function \cite{it}.
All of these dimensions share a unified sharp upper bound, $2r$, where $r=\vert\cS(\Lambda)\vert$.
\end{remark}
\begin{remark}\label{emre delta has projective then not cyclic algebra}
    Let $\Lambda$ be a cyclic Nakayama algebra. If $\Delta S\in\cB(\Lambda)$ has projective submodule, then $\bm\varepsilon(\Lambda)$ is not cyclic.
\end{remark}

\begin{lemma} Let $\Lambda$, $\Theta$ be Nakayama algebras such that $\bm\varepsilon(\Lambda)\cong\Theta$. If $\Theta$ has simple component, then $\domdim\Lambda\leq 2$.
    \end{lemma}
    \begin{proof}
        \begin{enumerate}
            \item $\Lambda$ is cyclic, so there exists $\Delta\in\cB(\Lambda)$ such that $\Hom_{\Lambda}(\cP,\Delta)$ is simple component of $\Theta$, $\Delta$ has no nontrivial extension with any other element of the base-set and there is a nonfiltered projective $\Lambda$-module $P$ such that $\Delta\subset P$.
            \item Now, $P/\Delta$ is submodule of an element of the base-set. Hence there exists an injective $\Lambda$-module $I$ such that $\Omega^1(I)\cong P/\Delta$, and $\Omega^2(I)\cong\Delta$. So, $\pdim I=2$, which forces that dominant dimension is at most two.
        \end{enumerate}
    \end{proof}

\subsection{Proofs of Theorems \ref{emre thm domdim},\ref{emre thm minimal aus. gor}}
The relationship between dominant dimensions of $\Lambda$ and $\Theta$ where $\bm\varepsilon(\Lambda)\cong\Theta$ is stated below. 
\begin{proposition}\label{emre may prop difference of domdims}
    Let $\Lambda$, $\Theta$ be Nakayama algebras such that $\bm\varepsilon(\Lambda)\cong\Theta$ and $\defect\Lambda=\defect\Theta\neq 0$. Then, for any injective non-projective $\Theta$-module $I$, there exists unique injective non-projective $\Lambda$-module $J$ such that $I\cong\Hom_{\Lambda}(\cP,\Omega^2(J))$. Moreover,
    \begin{enumerate}
        \item $\codomdim_{\Lambda} J=2+\codomdim_{\Theta} I$,
        \item $\pdim_{\Lambda}J=\pdim_{\Theta} I+2$.
        \end{enumerate}
\end{proposition}
\begin{proof}
    By proposition \ref{PROP every injective eps mod is second syzygy of injective lambda}, there exists exact sequence for every injective non-projective $\Theta$-module $I$ such that 
    \begin{align*}
        \Delta I\rightarrow P_1\rightarrow P_0\rightarrow J\rightarrow 0
    \end{align*}
    where $J$ is an injective $\Lambda$-module and $I\cong\Hom_{\Lambda}(\cP,\Omega^2(J))=\Hom_{\Lambda}(\cP,\Delta I)$. $P_0$ is projective-injective module since it is projective cover of injective module $J$. By lemma \ref{LEM injective first szyg is submodule of delta}, $\Omega^1(J)$ is not an element of $\cB(\Lambda)$, hence $P_1=P(\Omega^1(J))$ is projective $\Lambda$-module which is not $\cB(\Lambda)$-filtered. By proposition \ref{emre defect invariant prop}, $P_1$ is projective-injective and minimal projective. Therefore,
    \begin{gather*}
        \codomdim_{\Lambda}J=2+\codomdim_{\Lambda} \Delta I=
        2+\codomdim_{\Theta} I,\\
        \pdim_{\Lambda}J=2+\pdim_{\Lambda} \Delta I=
        2+\pdim_{\Theta} I.
    \end{gather*}
\end{proof}

Here, we give the proof of Theorem \ref{emre thm domdim}.
\begin{proposition}\label{emre proposition given theta lambda difference }
    Let $\Lambda$, $\Theta$ be Nakayama algebras such that $\bm\varepsilon(\Lambda)\cong\Theta$. Then,
   \begin{enumerate}
       \item  if $\defect\Lambda>\defect\Theta$, then $\domdim\Lambda\leq 2$.
       \item if $\defect\Lambda=\defect\Theta$, then $\domdim\Lambda=2+\domdim\Theta$.
       \end{enumerate}
\end{proposition}
\begin{proof}
    The first statement follows from the contra-positive of proposition \ref{PROP domdim geq 3 implies bijection}. For the second statement, we use proposition \ref{emre may prop difference of domdims}.
    \begin{align*}
        \codomdim\Lambda &=\inf\left\{\codomdim I\vert \,I\,\text{is injective}\right\}\\
        &=2+\inf\left\{\codomdim \Omega^2(I)\vert \,I\,\text{is injective}\right\}\\
        &=2+\inf\left\{\codomdim_{\Theta}(\Hom_{\Lambda}(\cP, \Omega^2(I))\vert \,I\,\text{is injective}\right\}\\
        &=2+\inf\left\{\codomdim_{\Theta}I'\vert \,I'\,\text{is injective}\,\Theta\text{-module}\right\}\\
        &=2+\codomdim\Theta
    \end{align*}
    Since co-dominant dimension and dominant dimension are equal by \cite{Mu}, we conclude that $\domdim\Lambda=2+\domdim\Theta$.
\end{proof}

\begin{remark}\label{emre may remark selfinjective}
   Let $\Lambda,\Theta$ be Nakayama algebras such that $\bm\varepsilon(\Lambda)\cong\Theta$ and $\defect\Lambda=\defect\Theta$. If $\Theta$ is selfinjective, then $\Lambda\cong\Theta$.  
\end{remark}
We prove Theorem \ref{emre thm minimal aus. gor}.

\begin{proposition}
   Let $\Lambda,\Theta$ be Nakayama algebras such that $\bm\varepsilon(\Lambda)\cong\Theta$ and $\defect\Lambda=\defect\Theta$. If $\Theta$ is dominant Auslander-Gorenstein algebra, $\Lambda$ is as well. 
\end{proposition}
\begin{proof}
By remark \ref{emre may remark selfinjective}, if $\Theta$ is selfinjective, then $\Lambda\cong\Theta$ is a dominant Auslander-Gorenstein algebra. If $\Theta$ is not selfinjective, then by proposition \ref{emre may prop difference of domdims}, for every $I$, there exists unique indecomposable injective $\Lambda$-module $J$ such that $\Hom_{\Lambda}(\cP,\Omega^2(J))\cong I$. Moreover,
\begin{align*}
     \codomdim_{\Lambda}J=2+\codomdim_{\Theta} I, \,\,       \pdim_{\Lambda}J=        2+\pdim_{\Theta} I.
\end{align*}
        $\Theta$ is a dominant Auslander-regular algebra, so every indecomposable injective $\Theta$-module $I$ satisfy $\pdim_{\Theta} I\leq \codomdim_{\Theta} I$.       Therefore, 
         \begin{align*}
             \pdim_{\Lambda} J=2+\pdim_{\Theta} I\leq 2+\codomdim_{\Theta} I=\codomdim_{\Lambda} J
         \end{align*}
         implies that $\Lambda$ is also dominant Auslander-Gorenstein algebra.
\end{proof}

\subsection{Dominant Auslander-Gorenstein Algebras among Nakayama Algebras}
Aim of this section is to prove the following:
\begin{proposition}
    Let $\Lambda$ be a cyclic Nakayama algebra of infinite global dimension. If $\Lambda$ is a dominant Auslander-Gorenstein algebra, then dominant dimension of $\Lambda$ is even.
\end{proposition}
\begin{proof}
We begin by considering the case where $\Lambda$ is self-injective, which implies a dominant dimension of zero. Thus, we can assume that $\Lambda$ is not self-injective.

Next, assuming the contrary, suppose that the dominant dimension of $\Lambda$ is $2k+1$. Consequently, the dominant dimension of $\bm\varepsilon^k(\Lambda)$ is one by remark \ref{emre remark reduction}. Moreover, $\bm\varepsilon^k(\Lambda)$ is a dominant Auslander-Gorenstein algebra by Proposition \ref{emre may prop difference of domdims}.

There exists an injective $\bm\varepsilon^k(\Lambda)$-module $I$ with $\pdim I = \codomdim I = 1$. Hence, $\Omega^1(I)$ is a projective module. By lemma \ref{every second syzy of injective}, there exists $\Delta\in\cB(\bm\varepsilon^k(\Lambda))$ with $\Omega^1(I)\subseteq\Delta$. According to Remark \ref{emre delta has projective then not cyclic algebra}, $\Delta$ is a projective module, which implies that the global dimensions of $\bm\varepsilon^k(\Lambda)$ and $\Lambda$ are finite, contradicting the assumption that $\gldim\Lambda=\infty$. Therefore dominant dimension of $\Lambda$ is even.
\end{proof}

\begin{remark}
    Let $\Theta$ be a minimal Auslander-Gorenstein algebra of infinite global dimension. Then, it is possible to construct a dominant Auslander-Gorenstein algebra $\Lambda$ such that $\bm\varepsilon(\Lambda)\cong\Theta$ and $\defect\Lambda\neq\defect\Theta$. We give an example. Let $\Lambda$ be given by Kupisch series $(4,4,3,3,2,2,2,2,2,2,2,3,4,4,4,4)$. One can verify that dominant dimensions of indecomposable projective non-injective modules are $2,4,4$. $\bm\varepsilon(\Lambda)$ is given by Kupisch series $(4,4,3,3,2,2,2,2,2,2,2,2,2,3)$ and it is $2$-Auslander-Gorenstein algebra with defect two.
\end{remark}

\subsection{Minimal Auslander-Gorenstein Algebras among Nakayama Algebras}\label{subsection minimal Auslande-Gorenstein}

In this subsection, we provide a classification of minimal Auslander-Gorenstein algebras. We recall that:
\begin{definition} \cite{iyamasolberg} An algebra $A$ is called minimal Auslander-Gorenstein if the Gorenstein dimension is bounded by the dominant dimension. $A$ is called $d$-Auslander-Gorenstein, if it is minimal Auslander-Gorenstein with injective dimension $d$.
\end{definition}

\begin{proposition}
    Let $\Lambda$ be a cyclic Nakayama algebra which is $d$-Auslander-Gorenstein where $d\geq 3$. Then, $\bm\varepsilon(\Lambda)$ is $(d-2)$-Auslander-Gorenstein algebra.
\end{proposition}

\begin{proof}
    Dominant dimension of $\Lambda$ is $d$, $d\geq 3$. So, by corollary \ref{cor of PROP domdim geq 3 implies bijection }, $\defect\Lambda=\defect\bm\varepsilon(\Lambda)$. By proposition \ref{emre proposition given theta lambda difference }, $\domdim\bm\varepsilon(\Lambda)=d-2$. By remark \ref{emre remark reduction}, $\gordim\bm\varepsilon(\Lambda)=d-2$, hence $\bm\varepsilon(\Lambda)$ is $(d-2)$-Auslander-Gorenstein algebra.
\end{proof}

When global dimension is finite, a minimal Auslander-Gorenstein algebra becomes a higher Auslander algebra, which we was studied in \cite{sen21}. Therefore, we consider Nakayama algebras of infinite global dimension. By \cite{sen18}, Gorenstein dimension of minimal Auslander-Gorenstein algebra is even. To give complete classification of cyclic Nakayama algebras which are minimal Auslander-Gorenstein, it is enough to construct $2$-Auslander-Gorenstein algebras by proposition above. 

We start with the following observation:
\begin{lemma}\label{emre auslandegor length of delta is at most two}
    Let $\Lambda$ be a $2$-Auslander-Gorenstein algebra of infinite global dimension. Then, any $\Delta_i\in\cB(\Lambda)$ satisfy $\ell(\Delta_i)\leq 2$. 
\end{lemma}
\begin{proof}

Assume that there exists at least one $\Delta_i$ module with a length greater than or equal to three. Denote the simple composition factors of such $\Delta_i$ as $X_1, \ldots, X_t$, where $t \geq 3$.

According to Lemma \ref{LEM the submodule of delta gives injective}, every proper radical of $\Delta_i$ is the first syzygy of an injective module. Notice that $\topp\Omega^2(I(X_j))\cong \tau\soc\Delta_i$ for any $1\leq j\leq t-1$.

Additionally, considering that every second syzygy of injective modules are projective due to the 2-Auslander-Gorenstein property of $\Lambda$, it follows that all $\Omega^2(I(X_i))$ are isomorphic to the same projective module. This implies that $P(\Omega^1(I(X_2)))$ is the radical of $P(\Omega^1(I(X_1)))$. Consequently, the former projective module cannot be injective.

Therefore, we can conclude that there exists an injective module with injective dimension two and dominant dimension one. As a result, $\Lambda$ cannot be a minimal Auslander-Gorenstein algebra.
     
\end{proof}

\begin{corollary}\label{emre auslandegor defect is at most one}
   Let $\Lambda$ be a $2$-Auslander-Gorenstein algebra of infinite global dimension. Then, defect of any indecomposable projective $\Lambda$-module can be at most one.
\end{corollary}
\begin{proof} 
    By the previous lemma, for any $\Delta_i\in\cB(\Lambda)$, we have $\ell(\Delta_i)\leq 2$. By remark \ref{socle of projective-injective}, $\Delta_i$ is a submodule of a projective-injective module $P$. If $\ell(\Delta_i)=1$, then defect of $P$ is zero, because the quotient $P/\Delta_i$ is radical of $P(\tau^{-1}S)$ where $S\cong\topp P$. If $\ell(\Delta_i)=2$, then defect of $P$ is one by lemma \ref{LEM the submodule of delta gives injective}.
\end{proof}

\begin{corollary} If $\Lambda$ is a minimal Auslander-Gorenstein algebra of infinite global dimension, then defect of projective modules are either one or zero.
\end{corollary}
\begin{proof}
    According to \cite{sen19}, $\Lambda$ has an even Gorenstein dimension, denoted as $2d$. Let $P$ be an indecomposable projective module with injective quotients $I_1, \ldots, I_t$, where $t \geq 3$. There exists a bijection between these injective modules and the injectives of the syzygy filtered algebra, established via the functor $\Hom_{\Lambda}(\mathcal{P}, \Omega^2(-))$. By induction, we obtain $t = \defect \Hom_{\Lambda}(\mathcal{P}, \Omega^{2d-2}(P))$. As $\bm\varepsilon^{2d-2}(\Lambda)$ is a $2$-Auslander-Gorenstein algebra, by Corollary \ref{emre auslandegor defect is at most one}, we conclude that $t \leq 2$.
\end{proof}
\subsubsection{Constrution of $2$-Auslander-Gorenstein Algebras among Nakayama Algebras}
We start with a selfinjective Nakayama algebra $\Theta$ and a complete set of representatives of indecomposable projective $\Theta$-modules $\{P_1,\dots,P_n\}$. 
Next we consider a sequence of integers $\cK=\left(c_1,\ldots,c_N\right)$ and show that the sequence is a Kupish series for a Nakayama algebra $\Lambda$ providing that it satisfies four conditions: 1), 2), 3), 4) as stated in the Proposition \ref{emre gorenstein description of K}. Before proving this proposition we will explain the role or purpose of each of these conditions.
\begin{enumerate}
    \item $N=n+d$ is a statement   that the number of indecomposable projective $\Lambda$-modules can be increased by any $1\leq d\leq 2n$, from the number of indecomposable projective $\Theta$-modules.
    \item $\cN=(c_{z_1},c_{z_2},\ldots,c_{z_n})$ will be the lengths of the filtered projective $\Lambda$-modules. These lengths are described in terms of integer valued function $f$ wih codomain $\{1,2\}$.
    \item This is the statement that the number of nonfiltered $\Lambda$-projectives is determined by function $f$.
    \item This is numerical condition which describes lengths of non-filtered $\Lambda$-projectives.
\end{enumerate}

\begin{proposition}\label{emre gorenstein description of K}
Let $\Theta$ be a selfinjective Nakayama algebra of rank $n$. Let $\left\{P_1,\ldots,P_n\right\}$ be the complete set of representatives of isomorphism classes of indecomposable projective $\Theta$-modules. Let $f$ be a function on simple $\Theta$-modules taking integer values either one or two. Consider the finite sequence of integers $\cK=\left(c_1,\ldots,c_N\right)$ satisfying
\begin{enumerate}[label=\arabic*)]
\item $N=n+d$ for some integer $d$ satisfying $1\leq d\leq 2n-1$
\item There exists a subsequence $\cN=(c_{z_1},c_{z_2},\ldots,c_{z_n})$ of $\cK$ such that
\begin{align}\label{emre gorenstein sunday lengths of N}
c_{z_i}=\sum\limits^{\ell (P_i)}_{j=1}f(S_{[j+i-1]}) 
\end{align}
where $[j]$ denote the least positive residue of $j$ modulo $n$ when $j>0$ and $[n]=n$, and $f(S_j)$ is integer valued function either one or two. 

\item Let $c_{z_i},c_{z_{i+1}}$ be consecutive elements of $\cN$ where $1\leq i\leq n-1$. Then, in $\cK$ the number of terms in between $c_{z_i}$ and $c_{z_{i+1}}$ is either zero or one. If $c_{z_i}$ and $c_{z_{i+1}}$ are consecutive in $\cK$, then $f(S_i)=1$. Else, $f(S_i)=2$. The number of terms between $c_{z_n}$ and $c_N$ is $b$ if and only if the number of terms between $c_1$ and $c_{z_1}$ is $1-b$.
\item  Any term $c_{a_m} \in\cK-\cN$ which satisfies ${z_i}<{a_m}<{z_{i+1}}$ is given by 
\begin{align}\label{emre gorenstein sunday lenghts of K-N}
c_{a_m}=1+\sum\limits^{\ell (P_{i+1})}_{j=1} f(S_{[j+i]})
\end{align} 
Any term $c_{a_m} \in\cK-\cN$ with index either $z_n<a_m\leq N$ or $1\leq a_m< z_1$   is given by 
\begin{align} \label{emre gorenstein sunday lengths o the rest}
c_{a_m}=1+\sum\limits^{\ell (P_{1})}_{j=1} f(S_{[j]})
\end{align} 
\end{enumerate}
The sequence $\cK$ satisfying all these conditions forms an admissible sequence of a cyclic Nakayama algebra $\Lambda$ which is $2$-Auslander-Gorenstein with defect $d$.
\end{proposition}

Before proving this, we need the following lemmas.

\begin{lemma}\label{emre goresntein lemma when difference is not one, all three possibilities here} Let $c_t=c_{z_i},c_{t+1}=c_{z_{i+1}}$. Then $c_t\leq c_{t+1}$ and $c_{t+1}-c_t=f(S_{[\ell(P_{i+1})]})-1$
\end{lemma}
\begin{proof}
The difference
\begin{gather*}
c_{z_i}-c_{z_{i+1}}=\sum\limits^{\ell (P_i)}_{j=1}f(S_{[j+i-1]}) -\sum\limits^{\ell (P_{i+1})}_{j=1}f(S_{[j+i]}) \\
=f(S_{[i]})+\sum\limits^{\ell (P_i)}_{j=2}f(S_{[j+i-1]})-\left(\sum\limits^{\ell (P_{i+1})-1}_{j=1}f(S_{[j+i]})+f(S_{[\ell(P_{i+1})]})\right)\\
=f(S_{[i]})+\sum\limits^{\ell (P_i)-1}_{j=1}f(S_{[j+i]})-\left(\sum\limits^{\ell (P_{i+1})-1}_{j=1}f(S_{[j+i]})+f(S_{[\ell(P_{i+1})]})\right)\\
=f(S_{[i]})+\left(\sum\limits^{\ell (P_i)-1}_{j=1}f(S_{[j+i]})-\sum\limits^{\ell (P_{i+1})-1}_{j=1}f(S_{[j+i]})\right)-f(S_{[\ell(P_{i+1})]})\\
=f(S_{[i]})-f(S_{[\ell(P_{i+1})]})=1-f(S_{[\ell(P_{i+1})]})\leq 0
\end{gather*}
since $\ell(P_i)=\ell(P_{i+1})$. $f(S_i)=1$, and $f(S_{[\ell(P_{i+1})]})$ is either one or two, hence the difference is less than or equal to zero.
\end{proof}

\begin{lemma}\label{emre gorenstein sunday lemma when difference is one} Let $c_t=c_{z_i}$, $c_{t+1}\notin\cN$. Then, $c_t-c_{t+1}=1-f(S_{[\ell(P_{i+1})]})\leq 0$.  Let $c_{t}\notin\cN$ and $c_{t+1}=c_{z_{i+1}}$. Then $c_t-c_{t+1}=1$.
\end{lemma}

\begin{proof}
The second statement follows from proposition \ref{emre gorenstein description of K} where we set the difference to one. For the first statement we have:
\begin{gather*} 
c_{z_i}-c_{t+1}=\sum\limits^{\ell (P_i)}_{j=1}f(S_{[j+i-1]})-\left(1+\sum\limits^{\ell (P_{i+1})}_{j=1}f(S_{[j+i]})\right) \\
=f(S_{[i]})+\sum\limits^{\ell (P_i)}_{j=2}f(S_{[j+i-1]})-\left(1+\sum\limits^{\ell (P_{i+1})-1}_{j=1}f(S_{[j+i]})+f(S_{[\ell(P_{i+1})]})\right)\\
=f(S_{[i]})+\sum\limits^{\ell (P_i)-1}_{j=1}f(S_{[j+i]})-\left(1+\sum\limits^{\ell (P_{i+1})-1}_{j=1}f(S_{[j+i]})+f(S_{[\ell(P_{i+1})]})\right)\\
=f(S_{[i]})+\left(\sum\limits^{\ell (P_i)-1}_{j=1}f(S_{[j+i]})-\sum\limits^{\ell (P_{i+1})-1}_{j=1}f(S_{[j+i]})\right)-1-f(S_{[\ell(P_{i+1})]})\\
=f(S_{[i]})-1-f(S_{[\ell(P_{i+1})]})=1-f(S_{[\ell(P_{i+1})]})
\end{gather*}
Since $f(S_i)=2$, the difference is less than or equal to zero.
\end{proof}

\begin{proof}
To verify whether a given sequence of integers $\cK=(c_1,\ldots,c_n)$ forms a Kupisch series of a cyclic Nakayama algebra, we need to check that 
\begin{enumerate}[label=\roman*)]
\item each $c_i\geq 2$
\item and if $c_i>c_{i+1}$, then $c_i-c_{i+1}=1$. 
\end{enumerate} If $c_i\leq c_{i+1}$, $c_{i+1}$ can take any positive integer value, so it satisfies admissibility. By lemmas \ref{emre gorenstein sunday lemma when difference is one} and \ref{emre goresntein lemma when difference is not one, all three possibilities here}, the case $c_j>c_{j+1}$ only happens when  $c_{t}\notin\cN$ and $c_{t+1}=c_{z_{i+1}}$. It follows that the difference is one, so $\cK$ is admissible sequence. In particular, each $c_j\in\cK$ has at least $\ell(P_{i})$-many nonzero summands. $\Lambda$ is a cyclic selfinjective Nakayama algebra, so $\ell(P_i)\geq 2$ for any $i$ which makes $c_j\geq 2$. This shows $\cK$ is an admissible sequence.
\end{proof}

Let $\left\{Q_1,\ldots,Q_{N}\right\}$ be the complete set of representatives of isomorphism classes of indecomposable $\Lambda$-modules and $\ell(Q_i)=c_i\in\cK$. We divide the proof into the following steps.

\begin{claim} $\defect\Lambda=d$ and $\vert\cS(\Lambda)\vert=n$
\end{claim}
\begin{proof}
By lemma \ref{sunday lemma lengths of delta and cardinality}, $\defect\Lambda=\vert\{c_{i+1}\mid c_i>c_{i+1}, 1\leq i\leq N\}\vert$ and $\vert\cS(A)\vert=\vert\{c_i\mid c_i\leq c_{i+1}, 1\leq i\leq N\}$ where $c_{N+1}=c_1$.  By lemmas \ref{sunday lemma when difference is not one, all three possibilities here} and \ref{sunday lemma when difference is one}, $c_i>c_{i+1}$ if $c_i\notin\cN$. There are exactly $d$ of them, so defect of $\Lambda$ is $d$, and it follows $\vert\cS(\Lambda)\vert=N-d=n$.
\end{proof}

\begin{claim} $\cB(\Lambda)$-filtered indecomposable projective $\Lambda$-modules are $Q_{z_1},\ldots,Q_{z_n}$.
\end{claim}

\begin{proof}
Notice that there are $n$-many $\cB(\Lambda)$-filtered indecomposable projective $A$-modules, since $\vert\cB(\Lambda)\vert=\vert\cS(\Lambda)\vert=n$. 
The lengths of elements of $\cB(\Lambda)$ are $c_{i+1}-c_i+1$ provided that $c_i\leq c_{i+1}$ by lemma \ref{sunday lemma lengths of delta and cardinality}. In lemmas \ref{sunday lemma when difference is one} and \ref{sunday lemma when difference is not one, all three possibilities here}, we calculated all possible differences of $c_i-c_{i+1}$ which are always of the form $1-f(S_j)$ when $c_i\leq c_{i+1}$. Moreover, by the construction in proposition \ref{emre gorenstein description of K}, for any $1\leq i\leq n$, $f(S_i)+1$ is a summand of $Q_{z_i}$. Therefore, every $Q_{z_i}$ is $\cB(\Lambda)$-filtered. Because there are $n$-many $Q_{z_i}$, hence we get all.  
\end{proof}

\begin{proof}
Now we can show that $\bm\varepsilon(\Lambda)\cong\Theta$. $\dim_{\mathbb{K}}\Hom_{\Lambda}\left(\bigoplus\limits_{1\leq i\leq n} Q_{z_i}, Q_{z_i}\right)=\ell_{\cB(\Lambda)}(Q_{z_i})=\ell(P_i)$, since the number of elements $\cB(\Lambda)$ appearing in each $Q_{z_i}$ is $\ell(P_i)$. Hence Kupisch series of $\bm\varepsilon(\Lambda)$ is equal to Kupisch series of $\Theta$ upto cyclic permutation.
\end{proof}

\begin{claim} $\Lambda$ is 2-Auslander-Gorenstein.
\end{claim}

\begin{proof}
Projective but not injective $\Lambda$ modules are $Q_i$ such that $\ell(Q_i)\in\cN$, $\ell(Q_{i-1})\in\cK-\cN$ and $\ell(Q_{i-2})\in\cN$. Consider the injective resolution of $Q_{i}$
\begin{align*}
0\rightarrow Q_i\rightarrow Q_{i-1}\rightarrow S_{i-1}\rightarrow 0.
\end{align*}
Notice that $S_{i-1}$ is simple proper submodule of an element of the base-set $\cB(\Lambda)$, since  $\coker g\in Filt(\cB(\Lambda))$ where $g:Q_i\rightarrow Q_{i-2}$ and $S_{i-1}\subset\coker g$. Therefore injective envelope of $S_{i-1}$ is projective-injective $\Lambda$-module, and by lemma \ref{injectivity} $\Sigma^2(Q_{i})\cong\Sigma^1(S_{i-1})$ is an injective non-projective $\Lambda$-module. Hence, injective dimension of $\Lambda$ is two. In particular, $Q_{i-1}$ and $I(S_{i-1})$ are projective-injective modules, which implies $\domdim\Lambda=2$.
\end{proof}

Now, we can apply the main theorem to $2$-Auslander-Gorenstein algebras to get all $2k$-Auslander-Gorenstein algebras $k\geq 2$. Here we give some examples. 
\begin{example} Let $\Theta$ be given by Kupisch series $(2,2)$. Then algebras given by $(3,2,2)$, $(5,5,4,4), (6,6,6,5,5)$ are minimal Auslander-Gorenstein algebras.
\end{example}

\begin{example}
    Let $\Theta$ be selfinjective algebra of rank $n$ and length $m$. 2-Auslander-Gorenstein algebra $\Lambda$ of defect $n$ which satisfies $\bm\varepsilon(\Lambda)\cong\Theta$ is given by the Kupisch series
    \begin{align*}
        ((2m,2m+1)^m)
    \end{align*}
\end{example}

\begin{example}
    Let $\Theta$ be selfinjective algebra of rank $n$ and length $m$. 2-Auslander-Gorenstein algebra $\Lambda$ of defect $n-1$ which satisfies $\bm\varepsilon(\Lambda)\cong\Theta$ is given by the Kupisch series
    \begin{align*}
        \left((2m,2m+1)^{n-m-1},2m,(2m,2m-1)^{m}\right)\quad & \text{if}\, n\geq m+1\\
        \left(f(P_1),1+f(P_2),f(P_2),\ldots,f(P_{n-1}),1+f(P_n),f(P_n)\right)\quad& \text{if}\, n<m
    \end{align*}
    where $f(P_i)=[P_i:S_n]+2\sum\limits_{j\neq n}[P_i:S_j]$.
\end{example}

\begin{example}
    Here we present a maple code to generate all 2-Auslander-Gorenstein algebras among Nakayama algebras from a given selfinjective Nakayama algebra $\Theta$.

\begin{verbatim}
# Takes kupisch series and gives ith projective object of Theta

proj := proc(length, order, rank) 
[seq(irem(j-1, rank)+1, j = order .. length+order-1))];end proc;       
\end{verbatim}\\

\begin{verbatim}
# calculates index of a simple module S in projective module P in Theta

weight := proc(P::list, S); 
    tot := 0; 
    for i from 1 to nops(P) do 
        if P[i] = S then tot := tot + 1   fi; 
    od; tot; end proc;
\end{verbatim}\\

\begin{verbatim}
 # creates length of ith filtered projective of Lambda
 for a given defect vector v
 
 newproj:=proc(i,m,n,v::list) 
 PP:=proj(m,i,n); tot:=0;
 for j from 1 to nops(v) do
 tot:=tot+weight(PP,v[j]);
 od; tot+m;end proc;
\end{verbatim}

\begin{verbatim}
# creates kupisch series of Lambda for a given defect vector v

kupischSeries:=proc(m,n,v::list);  
L:=[];
for i from 1 to n do
if member(i,v)=true then L:=[op(L),newproj(i,m,n,v),1+newproj(i+1,m,n,v)] 
else L:=[op(L),newproj(i,m,n,v)] fi;
od;end proc;
\end{verbatim}
\end{example}

\begin{example} Let $\Theta$ be a radical square zero selfinjective algebra. For vector $[1,1,0,\ldots,0]$, 2-Auslander-Gorenstein algebras are
\begin{gather*}
    [4, 5, 4, 5] \\
[4, 4, 3, 4, 3] \\
[4, 4, 3, 3, 2, 3] \\
[4, 4, 3, 3, 2, 2, 3] \\
[4, 4, 3, 3, 2, 2, 2, 3] \\
[4, 4, 3, 3, 2, 2, 2, 2, 3] \\
[4, 4, 3, 3, 2, 2, 2, 2, 2, 3] \\
[4, 4, 3, 3, 2, 2, 2, 2, 2, 2, 3] \\
[4, 4, 3, 3, 2, 2, 2, 2, 2, 2, 2, 3] \\
[4, 4, 3, 3, 2, 2, 2, 2, 2, 2, 2, 2, 3]
\end{gather*}
Let $\Theta$ be a selfinjective Nakayama algebra where $\ell(P_i)=5$ for all projetives.  For vector $(1,0,1,1,\ldots,0)$, all defect invariant algebras are
\begin{gather*}
[9, 9, 8, 9, 10, 9, 10] \\
[8, 9, 8, 8, 9, 8, 9, 8] \\
[8, 8, 7, 8, 8, 7, 8, 7, 8] \\
[8, 8, 7, 7, 8, 7, 7, 6, 7, 8] \\
[8, 8, 7, 7, 7, 6, 7, 6, 6, 7, 8] \\
[8, 8, 7, 7, 7, 6, 6, 5, 6, 6, 7, 8] \\
[8, 8, 7, 7, 7, 6, 6, 5, 5, 6, 6, 7, 8]
\end{gather*}
\end{example}

\subsection{Dominant Auslander-regular algebras among Nakayama Algebras}\label{subsection dominant Auslander}

\subsubsection{Construction of dominant Auslander-regular algebras of dominant dimension one}

\begin{lemma}
    Let $\Theta$ be a connected linear Nakayama algebra of rank $n$. Let $v_{\Theta}=(\defect P_1,\defect P_2,\ldots,\defect P_n)$ be its defect vector. Then, for any vector $w=(\defect P_1+m,\defect P_2,\ldots,\defect P_n)$ with $m\geq 1$, there exists cyclic Nakayama algebra $\Lambda$ such that $\bm\varepsilon(\Lambda)\cong\Theta$, $\defect\Lambda=\vert d\vert$ and $\domdim\Lambda=1$.
\end{lemma}

\begin{corollary}
    If $\Theta$ is dominant Auslander-regular algebra, then $\Lambda$ is as well.
\end{corollary}

\begin{proof} We apply proposition \ref{emre general proposition non defect invariant} for the given vector $w$ in the following way.
     Let $\defect P_1=t$. Hence, $\ell(\Delta(S_n))=m+t$. Let $X_1,\ldots,X_m,Y_1,\ldots,Y_t$ be the ordered composition factors of $\rad\Delta(S_n)$, with $\topp\rad\Delta(S_i)\cong X_1,\soc\rad\Delta(S_n)\cong Y_t$. 
     We choose the subsequence in part \ref{emre general length of unfiltered} of proposition \ref{emre general proposition non defect invariant} so that
     \begin{enumerate}
         \item $z_{n}<a_1<a_2<\cdots<a_{m+t}<z_1$,
         \item $b_1=a_{m+1}$, $b_2=a_{m+2}$, \ldots, $b_t=a_{m+t}$.
     \end{enumerate}
     This makes $P_{b_i}\cong P(Y_i)$ and each of them are projective-injective and minimal projective modules of $\Lambda$. For the remaining indices, the corresponding projective modules are set to be 
     \begin{align*}
         P(X_m)\subset P(X_{m-1})\subset \cdots\subset P(X_1)\subset P(\Delta(S_n)).
     \end{align*}
     By proposition \ref{emre general prop epsilon}, $\bm\varepsilon(\Lambda)\cong\Theta$.
     Notice that dominant dimension of $\Lambda$ is one, since each $P(X_j)$ is submodule of $P(\Delta(S_i))\in\cB(\Lambda)$, by lemma \ref{injectivity}, they are first syzygies of injective modules.
    \end{proof}

\subsubsection{Construction of dominant Auslander-regular algebras of dominant dimension two}
\begin{lemma}
    Let $\Theta$ be a connected linear Nakayama algebra of rank $n$. Let $v_{\Theta}=(\defect P_1,\defect P_2,\ldots,\defect P_n)$ be its defect vector. Then, for the vector $w=(\defect P_1,\defect P_2,\ldots,\defect P_n,1)$, there exists cyclic Nakayama algebra $\Lambda$ such that $\bm\varepsilon(\Lambda)\cong\Theta\oplus\mathbb{A}_1$, $\defect\Lambda=1+\defect\Theta$ and $\domdim\Lambda=2$.
\end{lemma}
\begin{proof}
    The vector $w$ describes lengths of elements of $\cB(\Lambda)$, i.e., $\ell(\Delta(S_i))\cong 1+\defect P_{i+1}$. In particular, $\ell(\Delta(S_n))=2$ and $\ell(\Delta(S_{n+1}))=1+\defect P_1$. We apply proposition \ref{emre general proposition non defect invariant}, and set $P(\Delta(S_{n+1}))$ as submodule of $P(\soc\Delta(S_n))$ which makes it projective-injective module. Dominant dimension of $\Lambda$ is two, since $\soc\Delta(S_n)$ is the first syzygy of an injective module, and its projective dimension is one. $\Lambda$ is a dominant Auslander-regular algebra, since for other injective objects we get $\codomdim_{\Lambda}J=\codomdim_{\Theta}I+2$.
\end{proof}
\begin{example}
    Let $\Theta$ be given by Kupisch series $(2,3,2,2,1)$. Then, $\Lambda$ is given by $(3,3,3,2,3,2,4,3)$ and $\bm\varepsilon(\Lambda)\cong\Theta\oplus\mathbb{A}$. Notice that projective module of length four corresponds to $P(\soc\Delta(S_5))$ in the proof above.
\end{example}

\subsubsection{Global dimension three} Here, we construct  cyclic Nakayama algebras of global dimension three, which are dominant Auslander-regular algebras.\\

Let $\Lambda$ be a cyclic Nakayama algebra of global dimension $3$. Then, $\gldim\bm\varepsilon(\Lambda)\leq 1$. Therefore $\bm\varepsilon(\Lambda)$ is a linear Nakayama algebra with components that are either linearly oriented $\mathbb{A}$-type or simple $\mathbb{A}_1$. Thus, we consider a Nakayama algebra $\Theta\cong\bm\varepsilon(\Lambda)$ such that 
\begin{align}\label{emre algebra linears simples}
    \Theta\cong \mathbb{A}_{n_1}\oplus\mathbb{A}_{n_2}\oplus\cdots\oplus\mathbb{A}_{n_k}
\end{align}
with $n_i\geq 1$, and at least one $n_j\geq 2$.

\begin{remark}
    By definition, defect of simple component is zero. However, in many cases it behaves like $\mathbb{A}_2$ algebra. That's why we define \emph{virtual defect} of simple components as one. Hence, if $n\geq 2$, then defect vector of $\mathbb{A}_n$ is $(n-1,0,\ldots,0)$ and defect vector of $\mathbb{A}_1$ is $(1)$. In \cite{sen21}, defect of simple components set to one. 
\end{remark}

\begin{proposition}\label{emre kupisch regular 1 and 3}
    Let $\Lambda$ be a cyclic Nakayama algebra given by Kupisch series
    \begin{align*}
        \left(n^{n_1}_1,n_1+x_1,n_1+x_1-1,\ldots,n_2+1,n^{n_2}_2,n_2+x_2,\ldots,n_3+1,n^{n_3}_3,n_3+x_3,\ldots,n_3+1,n_4+x_4,\ldots\right)
    \end{align*}
    where $x_i\geq n_{i+1}-1$ and $n_{i}\geq 2$. Then, $\Lambda$ is a dominant Auslander-regular algebra of global dimension three and dominant dimensions of projective modules are either three or one.
\end{proposition}
\begin{proof}
By direct computation, $\Lambda$ is non-defect invariant reverse of $\Theta\cong \mathbb{A}_{n_1}\oplus\mathbb{A}_{n_2}\oplus\cdots\oplus\mathbb{A}_{n_k}$. Hence global dimension of $\Lambda$ is three. Every projective module having proper injective quotient is of length $n_i+x_i$. Hence, projective and co-dominant dimensions of an injective $I$ with $\ell(I)\leq n_i+x_i-n_{i+1}$ is  one. If $\ell(I)>n_i+x_i-n_{i+1}$, then $0<\ell(\Omega^1(I))<n_{i+1}-1$. Hence $\ell(\Omega^2(I))<n_{i+1}+1$ which is not projective, and its projective cover is $P$ of length $n_{i+1}+x_{i+1}$. $\Omega^3(I)$ is projective by comparing lengths of $P$ and $\Omega^2(I)$. In particular, all projective covers are projective-injective, hence codominant dimension of $I$ is three. Hence $\Lambda$ is a dominant Auslander-regular algebra.
\end{proof}

\begin{proposition}\label{emre kupsich regular domdim 2 and 3}
    Let $\Lambda$ be a cyclic Nakayama algebra of global dimension three. If $\Lambda$ is dominant Auslander-regular algebra satisfying $\pdim I=3$ and $\pdim I'=2$, then it is given by
 \begin{gather*}
     ((1+g(n_1)+(g(n_2)-1)\delta(n_1,1))^{g(n_1)},n_1+g(n_2),n_1+g(n_2)-1,\ldots,1+g(n_2),\\(1+g(n_2)+(g(n_3)-1)\delta(n_2,1))^{g(n_2)},n_2+g(n_3),n_2+g(n_3)-1,\ldots,1+g(n_3)\\(1+g(n_3)+(g(n_4)-1)\delta(n_3,1))^{g(n_3)},n_3+g(n_4),n_3+g(n_4)-1,\ldots,1+g(n_4),\ldots\\\ldots\\\ldots
     (1+g(n_k)+(g(n_1)-1)\delta(n_k,1))^{g(n_k)},n_k+g(n_1),n_k+g(n_1)-1,\ldots,1+g(n_k))
\end{gather*}
where $\delta$ is Kronecker-delta function and 
\[
  g(n) =
  \begin{cases}
    n-1 & \text{if}\,\, n\geq 2 \\
    1 & \text{if}\,\, n=1
  \end{cases}
\]
\end{proposition}

\subsection{Higher Auslander Algebras among Nakayama Algebras}

\begin{proposition}\label{emre prop higher auslander gldim 3}Let $\Lambda$, $\Theta$ be Nakayama algebras such that $\bm\varepsilon(\Lambda)\cong\Theta$. $\Lambda$ is a higher Auslander algebra of global dimension three if and only if $\Theta\cong \mathbb{A}_{n_1}\oplus\mathbb{A}_{n_2}\oplus\cdots\oplus\mathbb{A}_{n_k} $ with $n_j\geq 2$. In this case, $\Lambda$ is given by Kupisch series
\begin{align*}
    \left(n^{n_1}_1,n_1+n_2-1,n_1+n_2-2,\ldots,n_2+1,n^{n_2}_2,n_2+n_3-1,\ldots,n_3+1,n^{n_3}_3,n_3+n_4-1,\ldots\right).
\end{align*}
\end{proposition}
\begin{proof}
 Global dimension of $\Lambda$ is three, therefore global dimension of $\bm\varepsilon(\Lambda)$ is one.
So, connected components of syzygy filtered algebra are either $\mathbb{A}$-type quivers or simple algebras, hence $\Theta\cong \mathbb{A}_{n_1}\oplus\mathbb{A}_{n_2}\oplus\cdots\oplus\mathbb{A}_{n_k}$. If at least one component is simple, i.e., $n_j=1$ for some $j$, then there is no defect invariant reverse of $\Theta$. By Theorem \ref{emre thm domdim} \ref{emre thm domdim p2} we would get $\domdim\Lambda\leq 2$, which contradicts to $\Lambda$ is a higher Auslander algebra of global dimension three. Therefore, each $n_j\geq 2$ . In this case, by proposition \ref{emre defect invariant prop}, we obtain Kupisch series of $\Lambda$.
\end{proof}

\begin{remark}
        Let $\Lambda$ be a cyclic Nakayama algebra given by Kupisch series as in proposition \ref{emre prop higher auslander gldim 3}. Then, $\Gamma:=\End_{\Lambda}(Q)^{op}$ is given by
    \begin{align*}\label{emre kupisch series of 2 cto}
        \left(n_1,n_1,n_1-1,\ldots,2,n_2,n_2,n_2-1,\ldots,2,n_3,n_3,n_3-1,\ldots,2,\ldots,n_k,n_k,n_k-1,\ldots,2\right)
    \end{align*}
    where $Q$ is projective-injective generator of $\Lambda$. $\Gamma$ has $2$-cluster-tilting object.
Moreover, any cyclic Nakayama algebra $\Gamma$ having $2$-cluster-tilting object $\cM$ such that $\cM\subseteq \Gamma\oplus D\Gamma\oplus \rad \Gamma $ is given by \ref{emre kupisch series of 2 cto}.
\end{remark}

\subsubsection{Global dimension four higher Auslander algebras} We give complete classification of cyclic Nakayama algebras of global dimension four which are higher Auslander. There are two cases we need to consider.
\begin{remark}
    Let $\Theta$ be an Auslander algebra. If $\Theta$ is cyclic, then it is given by Kupisch series $(2323,\ldots,23)$. If $\Theta$ is linear, then it is given by Kupisch series $((23)^k221)$ where $k\geq 0$. 
\end{remark}

\begin{proposition}
    Let $\Theta$ be given by Kupisch series $(23)^k$. Then, defect invariant reverse is given by Kupisch series $(344)^k$.
\end{proposition}
\begin{proof}
    By proposition \ref{emre defect invariant prop}, result follows.
\end{proof}

\begin{proposition}\label{emre prop higher auslander gldim 4}
Let $\mathbb{B}_k$ denote the algebra given by Kupisch series $((23)^k221)$. If $\Theta\cong \mathbb{B}_{n_1}\oplus  \mathbb{B}_{n_2}\oplus\cdots \oplus \mathbb{B}_{n_k} $, then the defect invariant reverse is given by Kupisch series
\begin{gather*}
    \left((344)^{n_1-1}333\right)^{\delta(n_1)}232,2+\delta(n_2),\left((344)^{n_2-1}333\right)^{\delta(n_2)}232,2+\delta(n_3),\ldots,\\\ldots,\left((344)^{n_k-1}333\right)^{\delta(n_k)}232,2+\delta(n_1))
\end{gather*}
where
\[
  \delta(n) =
  \begin{cases}
    0 & \text{if}\,\, n=0 \\
    1 & \text{if}\,\, n\neq 0
  \end{cases}
\]
\end{proposition}

\begin{proof}
    It follows from proposition \ref{emre defect invariant prop}.
\end{proof}
\begin{remark}
    Let $\Lambda$ be a cyclic Nakayama algebra given by Kupisch series as in proposition \ref{emre prop higher auslander gldim 4}. Then, $\Gamma:=\End_{\Lambda}(Q)^{op}$ is given by
    \begin{align}\label{emre kupisch series of 2 cto}
        \left(2(33)^{n_1},2,2,(33)^{n_2},\ldots,2(33)^{n_k}2\right)
    \end{align}
    where $Q$ is projective-injective generator of $\Lambda$. $\Gamma$ has $3$-cluster-tilting object.
Moreover, any cyclic Nakayama algebra $\Gamma$ having $3$-cluster-tilting object $\cM$ such that $\cM\subseteq \Gamma\oplus D\Gamma\oplus \rad \Gamma $ is given by \ref{emre kupisch series of 2 cto}.
\end{remark}

\section{The duality of syzygy filtered algebras }\label{sec:duality}

In this section, we reveal a duality between the syzygy filtered algebras $\bm\varepsilon(\Lambda)$ and the cosyzygy filtered algebra $\eta(\Lambda)$ for cyclic Nakayama algebra $\Lambda$. We will show equivalent definitions for these algebras first, and then establish a connection between them via minimal projective and minimal injective modules.

\subsection{Construction}
Let $\Lambda$ be a cyclic Nakayama algebra. Denote by $\Lambda\modd$ the category of finitely generated left $\Lambda$ modules, $\cP$ the subcategory of projective modules, $\cI$ the subcategory of injective modules and $\cQ$ the subcategory of projective-injective modules.   Assume $Q_i$ is an indecomposable projective-injective module, denote by $S_i$ the (unique) simple submodule of $Q_i$ and $T_i$ the unique simple quotient module of $Q_i$. Let $\cS=\{S_i\}$ and $\cT=\{T_i\}$.

Recall that a module is torsionless, if it is a submodule of a projective module. A simple module $S\in\cS$ if and only if $S$ is torsionless.
\begin{lemma-definition}\label{delta defn}
Let $S$ be a simple submodule of an indecomposable projective-injective module $Q$. Then there exists a unique largest submodule $\Delta(S)$ of $Q$ such that $\soc\Delta(S)=S$ and every quotient of $\Delta(S)$ is not torsionless.
\end{lemma-definition}
\begin{proof}
Assume $k>0$ is the smallest integer such that  $\tau^{-(k+1)}S$ is torsionless. Then $\Delta$ is the unique indecomposable submodule of $Q$ (of length $k$) with composition factors $S, \tau^{-1}S, \cdots, \tau^{-k}S$.  \qedhere
\end{proof}


The syzygy filtration technique was introduced by Sen in \cite{sen18}, \cite{sen19}, where the set of modules $\cB(\Lambda)=\{\Delta(S)| S\in\cS\}$ (called \emph{the base-set}) plays an important role in studying projective resolutions of modules over Nakayama algebras: Let $M$ be an indecomposable non-projective $\Lambda$ module. If its second syzygy is nontrivial, then $\Omega^2(M)$ has a unique filtration by $\cB(\Lambda)$. Therefore, the subcategory $\filt(\cB(\Lambda))$ consisting of modules with $\cB(\Lambda)$ filtrations controls second (and higher) syzygies of modules over Nakayama algebras. To understand the subcategory  $\filt(\cB(\Lambda))$, Sen introduced the following definition of syzygy filtered algebra:

\begin{definition}\label{defep} \cite{sen19}
For any cyclic Nakayama algebra $\Lambda$, the {\bf syzygy filtered algebra} $\bm\varepsilon(\Lambda)$ is the endomorphism algebra:
\begin{align*}
\bm\varepsilon(\Lambda):=\End_{\Lambda}\left(\bigoplus\limits_{S\in\cS}P(\tau S)\right).
\end{align*}
\end{definition}

The name of syzygy filtered algebra comes from the following result.

\begin{theorem}\cite{sen19}\label{filt equiv}
Denote by $\filt(\cB(\Lambda))$ the full subcategory of $\cB(\Lambda)$-filtered (right) $\Lambda$-modules and $\cP= \bigoplus\limits_{S\in\cS}P(\tau S)$. Let $\bm\varepsilon(\Lambda)\modd$ be the category of finitely generated right $\bm\varepsilon(\Lambda)$-modules. Then $\Hom(\cP,-):\filt(\cB(\Lambda))\to \bm\varepsilon(\Lambda)\modd$ is an equivalence.
\end{theorem}

Throughout this section, $\Lambda$ is a cyclic Nakayama algebra. Denote by $\cS=\{\soc Q| Q$ is an indecomposable projective-injective module$\}$ and $\cT=\{ \Top Q| Q$ is an indecomposable projective-injective module$\}$.

\subsection{The syzygy filtered algebra}

\begin{definition}\cite{sen19,rin2}
We call an indecomposable projective $\Lambda$-module $X$ minimal projective if $\rad X$ is not projective.  An indecomposable injective module $Y$ is minimal injective if $Y/\soc Y$ is not injective.	Denote by $\cX$ and $\cY$ the set of all the minimal projectives and all the minimal injectives respectively.
\end{definition}

\begin{proposition}\label{XY}
The set $\cX=\{P(\tau^-T)|T\in \cT\}$ and the set $\cY=\{I(\tau S)| S\in\cS\}$.
\end{proposition}

Realize that the Nakayama functor is a quasi-isomorphism between projectives and injectives. We have the following equivalent definition of the syzygy filtered algebras.

\begin{corollary}\label{eps inj}
The syzygy filtered algebra $\bm\varepsilon(\Lambda)\cong  \End_\Lambda\left(\bigoplus\limits_{S\in\cS}I(\tau S) \right)=\End_\Lambda\left(\bigoplus\limits_{Y\in\cY}Y \right)$
\end{corollary}

\subsection{The cosyzygy filtered algebra}
Dually, we can consider the cosyzygy filtered algebra.

\begin{lemma-definition}\label{nabla defn}
Let $T$ be the simple top of an indecomposable projective-injective module $Q$. Then there exists a unique largest quotient module $\nabla(T)$ of $Q$ such that $\Top\nabla(T)=T$ and every submodule of $\nabla(T)$ is not a quotient of an injective module.
\end{lemma-definition}
\begin{proof}
Similar to Lemma-Definition \ref{delta defn}.
\end{proof}


Let $\cB^\circ(\Lambda)=\{\nabla(T)| T\in\cT\}$ be the set of $\nabla$-modules, which are the composition factors of second (or higher) cosyzygy modules. i.e. for an indecomposable non-injective $\Lambda$ module $M$, if its second cosyzygy is nontrivial, then $\Omega^{-2}(M)$ has a unique filtration by $\cB^\circ(\Lambda)$.

\begin{definition}\label{defet}
For any cyclic Nakayama algebra $\Lambda$, the {\bf cosyzygy filtered algebra} $\bm{\eta}(\Lambda)$ is the endomorphism algebra:
\begin{align*}
\bm{\eta}(\Lambda):=\End_{\Lambda}\left(\bigoplus\limits_{T\in\cT}I(\tau^- T)\right)^{\op}.
\end{align*}
\end{definition}

Notice that we have to take the opposite algebra here, because of the following:

\begin{theorem}\label{cofilt equiv}
Denote by $\filt(\cB^\circ(\Lambda))$ the full subcategory of $\cB^\circ(\Lambda)$-filtered (right) $\Lambda$-modules and $\hat I= \bigoplus\limits_{T\in\cT}I(\tau^- T)$. Let $\eta(\Lambda)\modd$ be the category of finitely generated right $\eta(\Lambda)$ modules. Then $D\Hom(-,\hat I):\filt(\cB^\circ(\Lambda))\to \eta(\Lambda)\modd$ is an equivalence, where $D=\Hom(-,k)$.
\end{theorem}

\begin{proof}
	By definitions, the functor $D=\Hom(-,k)$ induces a duality: $\filt\cB^\circ(\Lambda)\to \filt\cB(\Lambda^{\op})$. It follows that $D\hat I$ is a projective generator of $\filt\cB(\Lambda^{op})$. Hence, by Theorem \ref{filt equiv}, $\Hom_{\Lambda^{\op}}(D\hat I,-):\filt\cB(\Lambda^{\op})\to \bm\varepsilon(\Lambda^{\op})\modd$ is an equivalence. On the other hand, $\bm\varepsilon(\Lambda^{\op})=\End_{\Lambda^{op}}(D\hat I)\cong\End_\Lambda(\hat I)=\eta(\Lambda)^{\op}$. In the following diagram:
	
	$$
	\xymatrix{
	\filt \cB^\circ(\Lambda)\ar[r]^D \ar[d]^F&\filt \cB(\Lambda^{\op}) \ar[d]^ {\Hom_{\Lambda^{\op}}(D\hat I, -)}\\
	\eta(\Lambda)\modd\ar[r]^D & \bm\varepsilon(\Lambda^{\op})\modd}
	$$

take $F=D\circ\Hom_{\Lambda^{\op}}(D\hat I, -)\circ D=D\Hom{\Lambda^{\op}}(D\hat I, D-)\cong D\Hom_\Lambda(-,\hat I)$, which is the desired equivalence functor.
\end{proof}

Due to Proposition \ref{XY}, we have the following equivalent definition of the cosyzygy filtered algebras.

\begin{corollary}\label{eta pr}
The cosyzygy filtered algebra $$\eta(\Lambda)\cong \End_{\Lambda}\left( \bigoplus\limits_{T\in\cT}P(\tau^-T)\right)^{\op}=\End_{\Lambda}\left( \bigoplus\limits_{X\in\cX}X\right)^{\op}.$$
\end{corollary}

\subsection{Connecting Syzygy and Cosyzygy Filtered Algebras }
Finally, we are going to show the duality between the syzygy filtered algebra and the cosyzygy filtered algebra. First, we need to fix a notation for some irreducible morphisms:

Given an indecomposable module $M$, denote by $i_M:\rad M\to M$ the embedding and $p_M: M\to M/\soc M$ the quotient map. Both $i_M$ and $p_M$ are unique up to isomorphism.

\begin{lemma}\label{truncate}
Let $f:M\to N$ be homomorphism between indecomposable modules. If $f$ is neither a monomorphism nor an epimorphism, then up to isomorphism $f$ can be uniquely decomposed as a composition: $f=i_N\circ f'\circ p_M$.	
\end{lemma}

\begin{proof}
Let $f:M\stackrel{\pi}\to \im f \stackrel{\rho}\to N$ be the canonical epi-mono decomposition. Since $f$ is not an epimorphism, $\rho$ factors through $\rad N$. i.e. $\rho=i_N\circ \rho'$. Since  $f$ is not a monomorphism either, $\pi$ factors through $M/\soc M$. i.e. $\pi=\pi'\circ p_M$. Hence set $f'=\rho'\circ\pi'$, it follows that $f=\rho\circ\pi= i_N\circ \rho'\circ \pi'\circ p_M=i_N\circ f'\circ p_M$. The uniqueness of the decomposition follows immediately from the fact that the monomorphism $i_N$ and the epimorphism $p_M$ are unique.
\end{proof}

Notice that if $M$ is a minimal injective or $N$ is a minimal projective, any non-isomorphism $f:M\to N$ is neither a monomorphism nor an epimorphism.

\begin{proposition}
For each minimal projective module $X$, there exists a unique minimal injective module $Y$ such that $Y/\soc Y\cong \rad X$. Furthermore, such a correspondence $X\mapsto Y$ is a bijection between $\cX$ and $\cY$.
\end{proposition}
\begin{proof}
See \cite{rin2} Appendix C.
\end{proof}

\begin{definition}
Define $\tilde\tau:\cX\to\cY$ a bijection as the following: for any minimal projective X, $\tilde\tau X$ is the unique module in $\cY$ such that $\tilde\tau X/\soc \tilde\tau X\cong \rad X$.
\end{definition}

\begin{lemma}\label{conn equiv}
The map $\tilde\tau$ induces an equivalence $F: \add\cX\to\add\cY$.
\end{lemma}
\begin{proof}
Let $F: \add\cX\to\add\cY$ be an additive functor, which on the objects $F(X)=\tilde\tau X$. Suppose there is a homomorphism $f: X_1\to X_2$, where $X_1$ and $X_2$ are minimal projectives. If $f$ is an isomorphism, then define $F(f):F(X_1)\to F(X_2)$ to be an isomorphism.  Otherwise, $f$ is neither a monomorphism nor an epimorphism. According to Lemma \ref{truncate}, there exists a unique decomposition $f=i_{X_2}\circ f'\circ p_{X_1}$. Here $f'\in\Hom_\Lambda(X_1/\soc X_1, \rad X_2)$ and both $X_1/\soc X_1$ and $\rad X_2$ are neither projective nor injective. Hence the Auslander-Reiten translation $\tau$ induces a bijection: $$\xymatrix{\Hom_\Lambda(X_1/\soc X_1, \rad X_2)\ar[r] &\Hom_\Lambda(\tau(X_1/\soc X_1), \tau(\rad X_2))\ar@{=}[d] \\&\Hom_\Lambda(FX_1/\soc FX_1, \rad FX_2)}$$
 Therefore, we can define $F(f)=i_{FX_2}\circ \tau f' \circ p_{FX_1}$.

$$
\begin{tikzpicture}
\node (1a) at (0,1) {$FX_1$};
\node (1c) at (2,1) {$X_1$};	
\node (2a) at (6,-1) {$FX_2$};	
\node (2c) at (8,-1) {$X_2$};	
\node (3a) at (4,-3) {$\im F(f)$};
\node (3b) at (6,-3) {$\im f$};

\draw (1a)--(1,0)--(1c);
\draw (2a)--(7,-2)--(2c);

\draw (1,0)--(3a)--(2a);
\draw (1c)--(3b)--(7,-2);
\end{tikzpicture}
$$

Next we want to show that $F$ is fully-faithful.

Let $f:X_1\to X_2$ be a non-isomorphism, where $X_1$ and $X_2$ are minimal projectives.
Assume $F(f)=0$. In particular, $f$ is not an isomorphism. Due to the definition, $F(f)=i_{FX_2}\circ \tau f' \circ p_{FX_1}$. So $0=\im F(f)=\im \tau f'= \tau\im f'= \tau \im f$, which implies $f=0$.

Let $g:FX_1\to FX_2$ be a homomorphism between minimal injectives. If $g$ is an isomorphism, $g=F(f)$ for some isomorphism $f: X_1\to X_2$. Otherwise, there is a decomposition $g=i_{FX_2}\circ g'\circ p_{FX_1}$, where $g':FX_1/\soc FX_1\to \rad FX_2$ and both $FX_1/\soc FX_1$ and $\rad FX_2$ are neither projective nor injective. Hence we can take $f=i_{X_2}\circ \tau^-g'\circ p_{X_1}$ and it is easy to check that $F(f)=g$.
\end{proof}

\begin{theorem}\label{dual thm}
Let $\Lambda$ be a cyclic Nakayama algebra and $\bm\varepsilon(\Lambda)$, $\eta(\Lambda)$ be its syzygy filtered algebra and cosyzygy filtered algebra. Then we have the following algebra isomorphism:
$$\bm\varepsilon(\Lambda)\cong\eta(\Lambda)^{\op}.$$
	
\end{theorem}

\begin{proof}
 $$\bm\varepsilon(\Lambda)\stackrel{Corollary \ref{eps inj}}\cong\End\left( \bigoplus\limits_{Y\in\cY}Y\right)\stackrel{ Lemma \ref{conn equiv}}\cong \End\left( \bigoplus\limits_{X\in\cX}X\right)\stackrel{Corollary \ref{eta pr}}\cong\eta(\Lambda)^{\op}.$$
\end{proof}

\subsection{A further question}
We should point out that the $\bm\varepsilon(\Lambda)$-$\eta(\Lambda)$ duality may exist in a more general context. Let $\Lambda$ be a finite dimensional algebra and $\hat Q$ the (basic) additive generator of projective-injective modules.  When $\domdim \Lambda\geq 1$, $\cT=\Gen(\hat Q)$ is a torsion class. Its corresponding wide subcategory (or thick subcategory) $\cU= \{X\in \cT: \forall (g:Y\to X)\in \cT,\, \ker(g)\in \cT\}$ is an abelian category and hence is equivalent to $\eta(\Lambda)\modd$ for some finite dimensional algebra $\eta(\Lambda)$. Such algebras have been studied by Asai \cite{A}. At the same time, $\cF=\Cogen(\hat Q)$ is a torsion-free class. Its corresponding wide subcategory $\cV= \{X\in \cF: \forall (g:X\to Y)\in \cF,\, \coker(g)\in \cF\}$ is an abelian category and hence is equivalent to $\epsilon(\Lambda)\modd$ for some finite dimensional algebra $\epsilon(\Lambda)$. It is still an open question to us whether $\bm\varepsilon(\Lambda)\cong\eta(\Lambda)^{\op}.$

\bibliographystyle{plain}

\begin{thebibliography}{99}


\bibitem{AIR}
T.~Adachi, O.~Iyama, and I.~Reiten.
\newblock $\tau$-tilting theory.
\newblock {\em Compos. Math.}, 150(3):415--452, 2014.

\bibitem{A}
S. Asai.
\newblock Semibricks.
\newblock {\em Int. Math. Res. Not}, 07 2018.

\bibitem{APT}
M. Auslander, M. Platzeck, G. Todorov,
\newblock{Homological theory of idempotent ideals}
\newblock{Transactions
of the American Mathematical Society, Volume 332, Number 2 , August 1992.}



\bibitem{aus}
M.Auslander, I. Reiten, S. Smalo.
\newblock Representation theory of Artin algebras.
\newblock {\em , Cambridge Studies in Advanced Mathematics, 1995}.



 \bibitem{AS}
M.~Auslander and S.~O. Smal{\o}.
\newblock Almost split sequences in subcategories.
\newblock {\em J. Algebra}, 69:426--454, 1981.

\bibitem{BCZ}
Emily Barnard, Andrew Carrol, Shijie Zhu.
\newblock Minimal inclusions of torsion classes.
\newblock {\em Algebraic Combinatorics}, 2(5), 879--901, 2019.

\bibitem{BTZ}
Emily Barnard, Gordana Todorov, Shijie Zhu.
\newblock Dynamical combinatorics and torsion classes.
\newblock Preprint available: arXiv:1911.10712, 2019.




\bibitem{CIM} A. Chan, O. Iyama, R. Marczinzik
\newblock Dominant Auslander-Gorenstein algebras and mixed cluster tilting.
\newblock {\em arxiv preprint 2210.06180}



\bibitem{CX}  H. Chen and C. Xi,
\newblock Dominant dimensions, derived equivalences and tilting modules,
\newblock Israel J. Math., 215 (2016), 349?C395.

\bibitem{DIRRT}
Laurent Demonet, Osamu Iyama, Nathan Reading, Idun Reiten, and Hugh Thomas.
\newblock Lattice theory of torsion classes.
\newblock Preprint available: arXiv:1711.01785, 03 2018.


\bibitem{gus}
W. Gustafson. Global dimension in serial rings. {\em Journal of Algebra 97.1, 14-16 1985}

\bibitem{it}
K. Igusa, G. Todorov. On the finitistic global dimension conjecture for artin algebras. {\em Representations of algebras and related topics, 201-204. Fields Inst. Commun., 45, Amer. Math. Soc., Province, RI, 2005}


\bibitem{iyama}
O. Iyama.
\newblock Auslander correspondence.
\newblock {\em Advances in Mathematics, 2007}.

\bibitem{iyamasolberg}
O. Iyama, O. Solberg. "Auslander-Gorenstein algebras and precluster tilting." Advances in Mathematics 326 (2018): 200-240.

\bibitem{rene}
D. Madsen,R. Marczinzik, G. Zaimi.
\newblock On the classification of higher Auslander algebras for Nakayama algebras.
\newblock {\em Journal of Algebra, 2020}.

\bibitem{MS}
Frederik Marks and Jan \v{S}\v{t}ov\'{\i}\v{c}ek.
\newblock Torsion classes, wide subcategories and localisations.
\newblock {\em Bull. Lond. Math. Soc.}, 49(3):405-416, 2017.

\bibitem{Mo} K.~Morita,
\newblock{Duality for modules and its applications in the theory of rings with minimum condition},
\newblock Sci. Rep. Tokyo Daigaku A, \textbf{6} (1958), 83-142.



\bibitem{Mu} B.J.~Mueller,
\newblock {The classification of algebras by dominant dimension},
\newblock Canadian Journal of Mathematics, \textbf{20} (1968), 398-409.



\bibitem{NRTZ}Van Nguyen, Idun Reiten, Gordana Todorov, Shijie Zhu,
\newblock Dominant dimension and tilting modules,
\newblock Mathematische Zeitschrift, 292: 947-973 2018, DOI: 10.1007/s00209-018-2111-4.

\bibitem{PS}
Matthew Pressland, Julia Sauter,
\newblock{Special tilting modules for algebras with positive dominant dimension}
\newblock{Preprint available: arxiv:1705.03367v2}, 2020.



\bibitem{rin1}
C. Ringel.
\newblock Representations of {$K$}-species and bimodules.
\newblock {\em Journal of Algebra, 1976}.

\bibitem{rin3}
Claus~Michael Ringel.
\newblock Reflection functors for hereditary algebras.
\newblock {\em J. London Math. Soc. (2)}, 21(3):465--479, 1980.

\bibitem{rin4}
Claus~Michael Ringel.
\newblock The {C}atalan combinatorics of the hereditary {A}rtin algebras.
\newblock In {\em Recent developments in representation theory}, volume 673 of
  {\em Contemp. Math.}, pages 51-177. Amer. Math. Soc., Providence, RI, 2016.



\bibitem{Ta} H.~Tachikawa,
\newblock On dominant dimension of QF-3 algebras,
\newblock Trans. Amer. Math. Soc., \textbf{112} (1964), 249--266.


\bibitem{rin2} 
CM. Ringel.
\newblock The finitistic dimension of a Nakayama algebra
\newblock {\em Journal of Algebra, 2021
}.

\bibitem{ringel2022linear}
CM. Ringel.
\newblock Linear nakayama algebras which are higher auslander algebras.
\newblock {\em Communications in Algebra}, 50(11):4842--4881, 2022.



\bibitem{sen18}
E. Sen.
\newblock The $\varphi$-dimension of cyclic nakayama algebras.
\newblock Communications in Algebra 49 (6), 2278-2299, 2021

\bibitem{sen19}
E. Sen.
\newblock Syzygy filtrations of cyclic nakayama algebras.
\newblock {\em arXiv preprint arXiv:1903.04645}, 2019.

\bibitem{sen21}
E. Sen.
\newblock Nakayama algebras which are higher auslander algebras.
\newblock {\em arXiv preprint arXiv:2009.03383}, 2020.

\bibitem{sen2021delooping}
E. Sen.
\newblock Delooping level of nakayama algebras.
\newblock {\em Archiv der Mathematik}, 117(2):141--146, 2021.

\end{thebibliography}

\end{document}